\documentclass[11pt,twoside]{article}

\usepackage[margin=1in]{geometry}

\usepackage{amsmath,amssymb,amsfonts,amsthm}
\usepackage{bm}
\usepackage{mathrsfs}
\numberwithin{equation}{section}
\allowdisplaybreaks

\usepackage{graphicx}
\usepackage{float}
\usepackage{array}
\usepackage{booktabs}
\usepackage{multirow}
\usepackage{makecell}

\usepackage{cite}
\usepackage{xcolor}
\usepackage[hidelinks]{hyperref}
\usepackage[nameinlink,capitalize,noabbrev]{cleveref}

\usepackage{fancyhdr}
\theoremstyle{plain}
\newtheorem{theorem}{Theorem}[section]
\newtheorem{lemma}[theorem]{Lemma}

\theoremstyle{definition}

\newtheorem{assumption}[theorem]{Assumption}
\newtheorem{example}[theorem]{Example}
\theoremstyle{remark}
\newtheorem{remark}[theorem]{Remark}

\newcommand{\vertiii}[1]{{\left\vert\kern-0.25ex\left\vert\kern-0.25ex\left\vert #1 
\right\vert\kern-0.25ex\right\vert\kern-0.25ex\right\vert}}

\hypersetup{
  pdftitle={Incremental SVD Compression for Nonlinear Oldroyd Equations with General Memory Kernels},
  pdfauthor={Gang Chen, Yangwen Zhang, and Dujin Zuo}
}

\title{Incremental SVD Compression for Nonlinear Oldroyd Equations with General Memory Kernels}
\author{Gang Chen \and Yangwen Zhang \and Dujin Zuo}
\date{\today}

\begin{document}

\maketitle

\begin{abstract}

We study mixed finite element/Crank--Nicolson discretizations of a nonlinear Oldroyd problem with general nonsingular and weakly singular memory kernels. Direct evaluation of the history term requires storing all previous velocity snapshots, which leads to $\mathcal{O}(mN)$ memory and $\mathcal{O}(mN^2)$ work over $N$ time steps, where $m$ denotes the number of spatial degrees of freedom. To reduce this burden, we compress the velocity history online by an incremental singular value decomposition and use the compressed representation in the discrete memory term. Under an approximate low-rank assumption of numerical rank $r$, the storage decreases to $\mathcal{O}((m+N)r)$, while the total history-evaluation work becomes $\mathcal{O}(mNr+rN^2)$. For nonsingular kernels, we derive a tolerance-dependent perturbation estimate showing that the baseline finite element accuracy is retained when the compression tolerance is sufficiently small. We also extend the approach to tempered weakly singular kernels via convolution quadrature. Numerical tests show near-indistinguishable solutions from the uncompressed scheme for the reported tolerances, together with substantial memory savings and reduced wall-clock time.

\end{abstract}

\begin{center}
\end{center}

\section{Introduction}

Nonlinear Oldroyd equations with memory terms arise naturally in the modeling of viscoelastic flows. After spatial and temporal discretization, the principal computational difficulty is the nonlocal history term: at each new time step, the solver must access a weighted combination of all previous velocity states. If the spatial discretization has $m$ degrees of freedom and the computation uses $N$ time steps, then straightforward history evaluation requires $\mathcal O(mN)$ storage and $\mathcal O(mN^2)$ total work. Reducing this history burden without sacrificing the accuracy of the underlying flow solver is the main objective of this paper.

\begin{subequations}\label{oldroyd-module}
	\begin{align}
		\bm{u}_t-\mu\Delta \bm{u}-\int_0^t \rho e^{-\delta(t-s)}\Delta \bm{u}(s){\rm d}s+(\bm{u}\cdot \nabla)\bm{u}+\nabla p&=\bm{f}, &\text{in } \Omega\times \left(0,T \right], \\
		\nabla \cdot\bm{u}&=0, &\text{in } \Omega\times \left(  0,T \right],\\
		\bm{u}&=\bm 0,&\text{on } \partial \Omega\times \left(  0,T \right],\\
		\bm{u}(\bm x,0)&=\bm{u}_0(\bm x),&\text{in } \Omega,
	\end{align}
\end{subequations}

Here $\bm u=\bm u(\bm x,t)$ is the velocity, $p=p(\bm x,t)$ is the pressure, $\bm f$ is a prescribed body force, $T$ is the final time, and $\Omega\subset\mathbb R^d$ $(d=2,3)$ is a bounded convex polygonal domain. A substantial numerical-analysis literature studies Oldroyd-type models and related nonlocal flow equations; see, for example, \cite{CannonJohnR.1999AmnG,LarssonStig1989TLBo,PaniAmiyaK.2005SfeG,AbbaszadehMostafa2020IotO,BirBikram2022BEmf,GuoYingwen2022Ceaf}. When $\rho=0$, \Cref{oldroyd-module} reduces to the incompressible Navier--Stokes equations, while the memory term introduces genuinely nonlocal-in-time dynamics.

Motivated by viscoelastic models with more general memory laws, we consider the following nonlinear Oldroyd problem with a general nonnegative memory kernel $K(t)$:

Find $ \bm{u}(\bm x,t) $ such that 
\begin{subequations}\label{OldryodModule-general}
	\begin{align}
		\bm{u}_t+\mathscr{A} \bm{u}+\int_0^t K(t-s)\mathscr{B} \bm{u}(s) \ {\rm d}s+(\bm{u}\cdot \nabla)\bm{u}+\nabla p&=\bm{f},&\text{in } \Omega\times \left(  0,T \right], \\
		\nabla \cdot\bm{u}&=0, &\text{in } \Omega\times \left(  0,T \right],\\
		\bm{u}&=\bm 0,&\text{on } \partial \Omega\times \left(  0,T \right] ,\\
		\bm{u}(\bm x,0)&=\bm{u}_0(\bm x),&\text{in } \Omega,
	\end{align}
\end{subequations}
where $\Omega \subset \mathbb{R}^d (d=2,3)$ is a bounded convex polygonal domain with Lipschitz boundary $\partial \Omega$, $ \bm u_0$ is a given function defined on $\Omega$, and $K(t)$ is either nonsingular or weakly singular. The function $\bm f$ is known, $\mathscr{A}$ is a symmetric positive definite second-order elliptic operator of the form
\begin{align}\label{ass-A}
	\mathscr{A}&=-\sum_{i, j=1}^d \frac{\partial}{\partial x_i}\left(a_{i j}(\bm x) \frac{\partial}{\partial x_j}\right)+a(\bm x) I,\qquad a(\bm x) \geq 0, \\
	a_{i j}(\bm x)=a_{j i}(\bm x), \qquad &a_1 \sum_{i=1}^d \xi_i^2 \geq \sum_{i, j=1}^d a_{i j} \xi_i \xi_j \geq a_0 \sum_{i=1}^d \xi_i^2,\qquad a_1 \geq a_0>0 ,
\end{align}
for all $\xi_i\in\mathbb R$, and $I$ stands for the $d\times d$ identity matrix. The operator $\mathscr{B}$ is a second-order linear operator of the form
\begin{align*}
	\mathscr{B}=-\sum_{i, j=1}^d \frac{\partial}{\partial x_i}\left(b_{i j}(\bm x) \frac{\partial}{\partial x_j}\right)+\sum_{i=1}^d b_i(\bm x) \frac{\partial}{\partial x_i}+b(\bm x) I.
\end{align*}

A large body of work reduces the cost of memory terms by exploiting special structure of the kernel, for example fast convolution quadrature, fast algorithms for fractional kernels, and related multistep or sum-of-exponentials techniques \cite{LubichC.1988Cqad,MR2231714,EduardoCuesta2006Cqtd,GuoLing2019Emmf,HuangYuxiang2022Eeot}. Our perspective is different. Instead of approximating the kernel, we compress the numerical \emph{solution history} itself. The key idea is to exploit approximate low rank in the matrix of velocity snapshots and update its low-rank representation online by incremental SVD. This viewpoint is closely related to online low-rank compression and incremental POD ideas for PDE data \cite{BrandMatthew2002ISVD,FareedHiba2018Ipod,FareedHiba2020Eaoa}, but here the compressed representation is used only inside the history term of the full-order Oldroyd solver.

Compared with our earlier work \cite{ChenGang2023AISM} on linear integro-differential equations, the current paper treats the nonlinear Oldroyd system, incorporates the mixed finite element structure and pressure variable, and extends the compression strategy to the weakly singular setting studied in Section~\ref{section 6}. The method is kernel-agnostic at the level of history compression, but its effectiveness depends on numerical low rank of the snapshot matrix.

\paragraph{Main contributions.}
The main contributions of this paper are as follows.
\begin{enumerate}
	\item We propose an online incremental-SVD compression strategy for the velocity history in mixed finite element/Crank--Nicolson discretizations of nonlinear Oldroyd equations with general memory kernels.
	\item Under a numerical rank-$r$ assumption, the method reduces history storage from $\mathcal O(mN)$ to $\mathcal O((m+N)r)$.
	\item For nonsingular kernels, we derive a tolerance-dependent perturbation estimate showing that the compressed scheme preserves the baseline finite element accuracy when the compression tolerance is sufficiently small.
	\item We extend the same compression strategy to tempered weakly singular kernels discretized by convolution quadrature.
	\item Numerical experiments show that, for the reported tolerances, the compressed solver is essentially indistinguishable from the uncompressed solver while requiring markedly less memory.
\end{enumerate}

The rest of the paper is organized as follows. Section~\ref{section 2} introduces the mixed finite element setting and the fully discrete scheme for nonsingular kernels. Section~\ref{sec3} recalls the incremental-SVD update used in the compression step. Section~\ref{section 4} integrates the compression procedure into the Oldroyd solver. Section~\ref{section 5} presents the nonsingular-kernel error analysis, Section~\ref{section 6} treats the weakly singular case, Section~\ref{section 7} reports numerical experiments, and Section~\ref{section 8} concludes with limitations and future directions.

\section{Fully discrete scheme for nonsingular kernels}\label{section 2}

In this section we introduce the mixed finite element setting and the fully discrete Crank--Nicolson approximation for \eqref{OldryodModule-general} in the nonsingular-kernel case. We first fix notation and then state the semidiscrete and fully discrete schemes.

\subsection{Preliminaries}
	For the mathematical setting of problems $ \eqref{OldryodModule-general} $, we use standard notations in Sobolev space with norm or seminorm and introduce the following space:
	\begin{align*}
		\bm V=(H_0^1(\Omega))^d,\ Q=L_0^2(\Omega),\ \bm{W}=\{\bm{w}\in \bm{ V}:\nabla\cdot\bm{w}=0\}.
	\end{align*} 
	Let $ \|\cdot\|_i $ and $ |\cdot|_i $ denote   the usual norm and seminorm defined on $ H^{i}(\Omega) $ or $ (H^i(\Omega))^d $ for any integer $ i $, and let $ (\cdot,\cdot)  $ represent $ L^2 $ inner product on $ \Omega $.  If $ i=0 $, we omit the subscript in $ \|\cdot\|_i $. 
	
	We define the bilinear operator $ \mathcal{A}(\cdot,\cdot),\mathcal{B}(\cdot,\cdot)$  on $ \bm V\times \bm V $, $\mathcal{D} $ on $ \bm V\times Q $ and trilinear operator $ \mathcal{C} $ on $ \bm V\times \bm V\times \bm V $ as follows:
	\begin{align*}
		\mathcal{A}(\bm u,\bm v)&=\sum_{i, j=1}^d\left(a_{i j}(\bm x) \frac{\partial \bm u}{\partial x_i}, \frac{\partial \bm v}{\partial x_j}\right)+(a(\bm x)\bm u,\bm v), \\
		\mathcal{B}(\bm u,\bm v)&=\sum_{i, j=1}^d\left(b_{i j}(\bm x) \frac{\partial\bm u}{\partial x_i}, \frac{\partial\bm v}{\partial x_j}\right)+\sum_{i=1}^d\left(b_i(\bm x) \frac{\partial\bm u}{\partial x_i},\bm v\right)+(b(\bm x)\bm u,\bm v),\\
		\mathcal{C}(\bm{u},\bm{v},\bm{w})&=\frac{1}{2}((\bm{u}\cdot\nabla)\bm{v},\bm{w})-\frac{1}{2}((\bm{u}\cdot\nabla)\bm{w},\bm{v}),\\
		\mathcal{D}(\bm{u},p)&=(\nabla \cdot \bm{u},p).
	\end{align*}
	Then we can derive the variational formula of $ \eqref{OldryodModule-general} $:\par Seek $ (\bm{u},p)\in \bm V\times Q $ such that
	\begin{align}\label{varitional-formula}
			(\bm{u}_t,\bm{v})+\mathcal{A}(\bm{u},\bm{v})+\mathcal{B}(\int_0^t K(t-s)\bm{u}(s){\rm d}s,\bm{v})&+\mathcal{C}(\bm{u},\bm{u},\bm{v})	-\mathcal{D}(\bm{v},p)=(\bm{f},\bm{v}),\\
			\mathcal{D}(\bm{u},q)&=0,\nonumber
	\end{align}
	for any $ (\bm{v},q)\in \bm V\times Q. $
	 Equivalently, find $ \bm{u}\in \bm{W} $ such that
	 \begin{align}
	 	(\bm{u}_t,\bm{v})+\mathcal{A}(\bm{u},\bm{v})+\mathcal{B}(\int_{0}^{t}K(t-s)\bm{u}(s){\rm d}s,\bm{v})+\mathcal{C}(\bm{u},\bm{u},\bm{v})=(\bm{f},\bm{v}),\ \forall \bm{v}\in \bm{W}, \ t>0.
	 \end{align}
	Here are some classical estimates of the trilinear form $ \mathcal{C}(\cdot,\cdot,\cdot) $, which are essential for the stability and convergence analysis of  system $ \eqref{OldryodModule-general} $.
	\begin{lemma} 
		\cite[Lemma 2.1]{GuoYingwen2022Ceaf}\label{trilinear-property}
		There exists a positive constant $ \gamma_0 $, depending only on $ \Omega $, such that :
		\begin{align}
			&\mathcal{C}(\bm{u},\bm{v},\bm{w})=-\mathcal{C}(\bm{u},\bm{w},\bm{v}),\ \forall \ \bm{u},\bm{v},\bm{w}\in \bm V,\\
			&\mathcal{C}(\bm{u},\bm{v},\bm{v})=0,\ \forall \ \bm{u}, \bm{v}\in \bm V,\\
			&|\mathcal{C}(\bm{u},\bm{v},\bm{w})|\le \gamma_0\|\nabla\bm{u}\|_0\|\nabla\bm{v}\|_0\|\nabla\bm{w}\|_0,\ \forall \ \bm{u},\bm{v},\bm{w}\in \bm V,\\
			&|\mathcal{C}(\bm{u},\bm{v},\bm{w})|+|\mathcal{C}(\bm{w},\bm{u},\bm{v})|\le \gamma_0(\|\bm{u}\|^{1/2}\|\nabla\bm{u}\|^{1/2}\|\nabla\bm{v}\|\notag\\
			&\qquad+\|\bm{v}\|^{1/2}\|\nabla\bm{v}\|^{1/2}\|\nabla\bm{u}\|)\|\bm{w}\|^{1/2}\|\nabla\bm{w}\|^{1/2},\ \forall \ \bm{u},\bm{v},\bm{w}\in \bm V.
		\end{align}
	\end{lemma}
	
\subsection{Mixed finite element setting}

Let $ \bm V_h\subset \bm V $ and $ Q_h\subset Q $ be conforming finite element spaces on a regular simplicial triangulation $\mathcal T_h$ of $\Omega$. The discrete divergence-free space is defined by
\begin{align*}
	\bm{W}_h=\{\bm{v}_h\in\bm{V}_h:\mathcal D(\bm v_h,q_h)=0,\ \forall \ q_h\in Q_h\}.
\end{align*}
We assume that the pair $(\bm V_h,Q_h)$ satisfies the discrete inf-sup condition: there exists a constant $\beta_0>0$, independent of $h$, such that
\begin{align*}
	\inf_{0\neq q_h\in Q_h}\sup_{0\neq \bm v_h\in \bm V_h}
	\frac{\mathcal D(\bm v_h,q_h)}{\|\bm v_h\|_1\|q_h\|}\ge \beta_0.
\end{align*}
Typical stable choices include the Mini element and the Taylor--Hood family; the numerical experiments in Section~\ref{section 7} use the Mini element. For the approximation argument, we introduce a divergence-preserving projection $\bm P_h:\bm W\cap (H^{k+1}(\Omega))^d\to \bm W_h$ and the $L^2$-projection $\rho_h:Q\to Q_h$, and assume that there exist integers $k,l\ge 1$ and a constant $\gamma_1>0$, independent of $h$, such that
\begin{align*}
	\|\bm P_h\bm v-\bm v\|+h|\bm P_h\bm v-\bm v|_1&\le \gamma_1h^{k+1}|\bm v|_{k+1},\qquad \forall \bm v\in \bm W\cap (H^{k+1}(\Omega))^d,\\
	\|\rho_h q-q\|+h|\rho_hq-q|_1&\le \gamma_1h^{l+1}|q|_{l+1},\qquad \forall q\in Q\cap H^{l+1}(\Omega).
\end{align*}
Such approximation properties are standard for stable mixed pairs; see, for example, \cite{GiraultRaviart1986}.

Then the semidiscrete Galerkin scheme for \eqref{OldryodModule-general} can be expressed as follows: find $(\bm{u}_h,p_h)\in \bm V_h\times Q_h$ such that, for all $t>0$ and all $(\bm{v}_h,q_h)\in \bm V_h\times Q_h$,

Then the  semidiscrete Galerkin scheme for the system \eqref{OldryodModule-general} can be expressed as follows: find $ (\bm{u}_h, p_h) \in \bm V_h\times Q_h$ such that for all $ t>0 $, $ (\bm{v}_h,q_h)\in \bm V_h\times Q_h  $:
	\begin{align}
		\begin{split}
			(\bm{u}_{h,t},\bm{v}_h)&+
			\mathcal{A}(\bm{u}_h,\bm{v}_h)
			-\mathcal{D}(\bm{v}_h,p_h)+\mathcal{D}(\bm{u}_h,q_h)\\
			+\mathcal{B}(\int_0^t&K(t-s)\bm{u}_h\ {\rm{d}}s,\bm{v}_h)+\mathcal{C}(\bm{u}_h,\bm{u}_h,\bm{v}_h)=(\bm{f},\bm{v}_h),\\
			&\bm{u}_h(0)=\bm{u}_{h}^0,
		\end{split}
	\end{align}
	where $\bm u_{h,t} $ denotes the time derivative of $ \bm u_h $ and  $\bm u_{h}^0 $ is some type of  projection of $ \bm{u}_0 $ onto space $ \bm V_h $, which will be specified later.

	\subsection{Fully discrete scheme}\label{Fully discrete scheme}

We now discretize \eqref{OldryodModule-general} in time by the Crank--Nicolson method for a nonsingular kernel $K$. Let $\Delta t=T/N$, let $t_n=n\Delta t$, and denote $\bar t_n=(t_n+t_{n-1})/2$. For the memory term we use a midpoint-type quadrature, following \cite{GuoYingwen2022Ceaf}. The resulting fully discrete scheme reads as follows:

	find $ (\bm{u}_h^{n},p_h^{n})\subset \bm V_h\times Q_h $  for $ n=1,2,\cdots,N $, such that
	\begin{align}\label{fully-discretization--nonsingular}
		\begin{split}
			(d_t\bm{u}_h^{n},\bm{v}_h)+
			\mathcal{A}(\bar{\bm{u}}_h^{n},\bm{v}_h)
			+\mathcal{B}(\bar{\bm u}_{K,h}^{\Delta t,n},\bm{v}_h)
			-&\mathcal{D}(\bm{v}_h,\bar p_h^{n})
			+\mathcal{C}(\bar{\bm{u}}_h^n,\bar{\bm{u}}_h^{n},\bm{v}_h)=(\bm{f}(\bar t_n),\bm{v}_h),\\
			&\mathcal{D}(\bar{\bm{u}}_h^{n},q_h)=0,
		\end{split}
	\end{align}
	where
	\begin{align*}
		&d_t\bm{u}_h^n=\frac{\bm{u}_h^n-\bm{u}_h^{n-1}}{\Delta t},\ 
		\overline{\bm{u}}_h^n=\frac{\bm{u}_h^n+\bm{u}_h^{n-1}}{2},\ 
		\bar t_n=\frac{t_n+t_{n-1}}{2},\\
		&\bar{\bm{u}}_{K,h}^{\Delta t,n}=\Delta t\sum_{j=1}^{n-1}K(\bar{t}_n-\bar{t}_j)\bar{\bm{u}}_h^j+\frac{\Delta t}{2}K(0)\bar{\bm{u}}_h^n,
	\end{align*}
and $\bm{u}_h^0=\bm P_h\bm u_0$.

\begin{remark}

A discrete initial pressure may be obtained from a standard projection of the momentum equation at $t=0$. Since this initialization is routine and does not affect the history-compression analysis, we omit the formula here.

\end{remark}

At time level $n+1$, the discrete memory term requires access to the previously computed velocity states $\{\bm u_h^j\}_{j=0}^{n}$. If $\bm V_h=\mathrm{span}\{\bm\phi_j\}_{j=1}^{m}$ and the coefficient vector of $\bm u_h^n$ in this basis is denoted by $\bm u_n\in\mathbb R^m$, then straightforward history storage costs $\mathcal O(mn)$, while direct accumulation of the history term over $n$ steps costs $\mathcal O(mn^2)$. Hence the standard scheme is linear in storage but quadratic in the number of time steps for history evaluation. The goal of the compression strategy below is to reduce the dependence on the full spatial dimension $m$ by exploiting approximate low rank of the snapshot matrix.

\section{Incremental SVD preliminaries}\label{sec3}

This section recalls the incremental SVD algorithm used to compress the velocity snapshot matrix online. The relevant structural assumption is \emph{numerical} low rank: for the tolerance $\mathtt{tol}$ used in the compression step, the singular values of the snapshot matrix decay sufficiently rapidly that only a small number $r\ll \min\{m,n\}$ need to be retained. In contrast to kernel-structured fast methods, the present approach is therefore kernel-agnostic but depends on empirical singular-value decay of the computed solution history.

Given a vector $u \in \mathbb{R}^m$ and an integer $r$ satisfying $r \leq m$, the notation $u(1:r)$ denotes the first $r$ components of $u$. Likewise, for a matrix $U \in \mathbb{R}^{m \times n}$, the notation $U(p:q,r:s)$ denotes the submatrix formed by rows $p,\ldots,q$ and columns $r,\ldots,s$. Throughout this section, $|\cdot|$ denotes the Euclidean norm on $\mathbb R^m$.

Assume that after processing the first $\ell$ columns of the data matrix $U=[u_1|\cdots|u_n]$, we have a rank-$k$ truncated SVD
\begin{align}\label{eq201}
	U_\ell \approx Q\Sigma R^\top,\qquad Q^\top Q=I_k,\qquad R^\top R=I_k,\qquad \Sigma=\mathrm{diag}(\sigma_1,\ldots,\sigma_k),
\end{align}
with $k\le r$. We summarize the update in four steps.

\subsection*{Step 1: Initialization}
				Assuming that the first column of matrix $U$, denoted as $u_1$, is non-zero, we can proceed to initialize the SVD of $u_1$ using the following approach:
				\begin{align*}
					\Sigma=(u_1^\top u_1)^{1/2},\qquad Q=u_1\Sigma^{-1},\qquad R=1.
				\end{align*}

				Assuming we already have  the  truncated SVD of rank $k$ for the first $\ell$ columns of matrix $U$, denoted as $U_{\ell}$:
				\begin{align}
					U_\ell \approx Q\Sigma R^\top,\quad \textup{with} \quad Q^\top Q=I_k,\quad R^\top R=I_k,\quad \Sigma= \texttt{diag}(\sigma_1,\cdots,\sigma_k),
				\end{align}
				where $ \Sigma \in \mathbb{R}^{k\times k}$ is a diagonal matrix with the $k$ ordered singular values of $ U_\ell$ on the diagonal, $ Q\in \mathbb{R}^{m\times k} $ is the matrix of the corresponding $k$ left singular vectors of $ U_\ell $ and $ R\in\mathbb{R}^{\ell\times k} $ is the matrix of the corresponding $ k $ right singular vectors of $ U_\ell $.
				
				Given our assumption that the matrix $U$ is low rank, it is reasonable to expect that most of the columns of $U$ are either linearly dependent or nearly linearly dependent on the vectors in $Q \in \mathbb{R}^{m \times k}$. Without loss of generality, we assume that the next $s$ vectors, denoted as $\left\lbrace u_{\ell+1}, \ldots, u_{\ell+s}\right\rbrace$, their residuals are less than a specified tolerance when projected onto the subspace spanned by the columns of $Q$. However, the residual of $u_{\ell+s+1}$ is larger than the given tolerance. In other words,
				\begin{subequations}
					\begin{align}
						|u_i-QQ^\top u_i  | &< \texttt{tol}, \quad i=\ell+1,\cdots, \ell+s,\label{lessthantol}\\
						|u_i-QQ^\top u_i   |  &\ge \texttt{tol}, \quad i=\ell+s+1.\label{largerthantol}
					\end{align}
				\end{subequations}
				The symbol $|\cdot|$ denotes the Euclidean norm within the realm of $\mathbb{R}^m$.
				
				\subsection*{Step 2: Update the SVD of $ U_{\ell+s} $ ($p$-truncation)}
				By  the  assumption \eqref{lessthantol}, we have 
				\begin{align*}
					U_{\ell+s}&= \left[ U_\ell\mid u_{\ell+1}\mid\cdots\mid u_{\ell+s}\right] \\
					&\approx \left[ Q\Sigma R^\top \mid u_{\ell+1}\mid\cdots\mid u_{\ell+s}\right] \\
					&\approx \left[ Q\Sigma R^\top \mid QQ^\top u_{\ell+1}\mid\cdots\mid QQ^\top u_{\ell+s}\right] \\
					&=Q \underbrace{\left[ \Sigma\mid Q^\top u_{\ell+1}\mid\cdots\mid Q^\top u_{\ell+s}\right] }_{Y}\left[\begin{array}{cc}
						R & 0 \\
						0 & I_s
					\end{array}\right]^\top.  
				\end{align*}    
				
				We can obtain the truncated SVD of $U_{\ell+s}$ by computing the thin SVD of the matrix $Y$. Specifically, let $Y = Q_Y \Sigma_Y R_Y^\top$ be the SVD of $Y$, and split $ R_Y $ into $ \left[\begin{array}{cc}
					R_Y^{(1)} \\
					R_Y^{(2)}
				\end{array}\right]$. With this, we can update the SVD of $U_{\ell+s}$ as follows:
				\begin{align*}
					Q\leftarrow QQ_Y,\quad \Sigma\leftarrow\Sigma_Y,\quad  R\leftarrow \left[\begin{array}{cc}
					 RR_Y^{(1)} \\
					 R_Y^{(2)}
					\end{array}\right] \in \mathbb{R}^{(\ell+s)\times k}.
				\end{align*}
				
				It is worth noting that the dimensions of the matrices $Q$ and $\Sigma$ remain unchanged, and we need to incrementally store the matrix $W = \left[ Q^\top u_{\ell+1}\mid\cdots\mid Q^\top u_{\ell+s}\right]$. As $W$ belongs to $\mathbb{R}^{k\times s}$ where $k\leq r$ is relatively small, the storage cost for this matrix  is low.   
				
				\subsection*{Step 3: Update the SVD of $ U_{\ell+s+1} $ (No truncation)}
				Next, we proceed with the update of the SVD for $U_{\ell+s+1}$. Firstly, we compute the residual vector of $u_{\ell+s+1}$ by projecting it onto the subspace spanned by the columns of $Q$, i.e.,
				\begin{align}\label{residual}
					e=u_{\ell+s+1}-QQ^\top  u_{\ell+s+1}.
				\end{align}
				
				First, we define $p = |e|$. Then, based on \eqref{largerthantol}, we deduce that $p > \textup{\texttt{tol}}$. Finally, we denote $\widetilde{e}$ as $e/p$.  With these definitions, we establish the following fundamental identity:
				\begin{align*}
					U_{\ell+s+1}&= \left[U_{\ell+s}\mid u_{\ell+s+1} \right] \\
					&\approx \left[ Q\Sigma R^\top \mid p\widetilde{e}+QQ^\top u_{\ell+s+1}\right] \\
					&\approx [Q \mid \widetilde{e}] \underbrace{\left[\begin{array}{cc}
							\Sigma & Q^{\top} u_{\ell+s+1} \\
							0 & p
						\end{array}\right]}_{\bar{Y}}\left[\begin{array}{cc}
						R & 0 \\
						0 & 1
					\end{array}\right]^{\top}.
				\end{align*}
				
				Let $ \bar{Q}\bar{\Sigma}\bar{R}^\top  $  be the full SVD of $\bar Y$. Then  the SVD of $ U_{\ell+s+1} $ can be approximated by
				\begin{align*}
					U_{\ell+s+1}  \approx (\left[ Q\mid \widetilde{e}\right]\bar{Q} )\bar{\Sigma} \left(\left[\begin{array}{cc}
						R & 0 \\
						0 & 1
					\end{array}\right]\bar{R}\right)^\top.
				\end{align*}
				
				With this, we can update the SVD of $U_{\ell+s+1}$ as follows:
				\begin{align*}
					Q\leftarrow (\left[ Q\mid \widetilde{e}\right])\bar{Q}, \quad \Sigma\leftarrow\bar{\Sigma}, \quad  R\leftarrow \left[\begin{array}{cc}
						R & 0 \\
						0 & 1
					\end{array}\right]\bar{R}.
				\end{align*}
				It is worth noting that, in this case, the dimensions of the matrices $Q$ and $\Sigma$ increase.
				\begin{remark}
					Theoretically, the residual vector $e$ in \eqref{residual} is orthogonal to the vectors in  the subspace spanned by the columns of $Q$. However, in practice, this orthogonality can be completely lost, a fact that has been confirmed by numerous numerical experiments \cite{FareedHiba2020Eaoa,FareedHiba2018Ipod,OxberryGeoffreyM.2017Lass}.  In \cite{GiraudL.2005Tloo}, Giraud et al. stressed that exactly two iteration-steps are enough to keep the orthogonality. To reduce computational costs, Zhang  \cite{ZhangYangwen2022Aata} suggested using the two iteration steps only when the inner product between $e$ and the first column of $Q$   exceeds a certain tolerance. Drawing from our experience, it is imperative to calibrate this tolerance to align closely with the machine error. For instance, as demonstrated in this paper, we consistently establish this tolerance as $10^{-14}$.
				\end{remark}
				\subsection*{Step 4: Singular value  truncation}
				For many  PDE data sets, they may have a large number of nonzero singular values but most of them are very small. Considering the computational cost involved in retaining all of these singular values, it becomes necessary to perform singular value truncation. This involves discarding the last few singular values if they fall below a certain tolerance threshold.
				\begin{lemma}\cite[Lemma 5.1]{ZhangYangwen2022Aata}
					Assume that $\Sigma=\operatorname{diag}\left(\sigma_1, \sigma_2, \ldots, \sigma_k\right)$ with $\sigma_1 \geq \sigma_2 \geq \ldots \geq \sigma_k$, and $\bar{\Sigma}=\operatorname{diag}\left(\mu_1, \mu_2, \ldots, \mu_{k+1}\right)$ with $\mu_1 \geq \mu_2 \geq \ldots \geq \mu_{k+1}$. Then we have
					\begin{align}
						\label{eq302}
						\mu_{k+1} &\le p, \\
						\label{eq303}
						\mu_{k+1} &\le \sigma_k \leq \mu_k \leq \sigma_{k-1} \leq \ldots \leq \sigma_1 \leq \mu_1.
					\end{align}
					
				\end{lemma}
				
				The inequality $\eqref{eq302}$ indicates that, regardless of the magnitude of $p$, the last singular value of $\bar{Y}$ can potentially be very small. This implies that the tolerance set for $p$ cannot prevent the algorithm from computing exceedingly small singular values. Consequently, an additional truncation is necessary when the data contains numerous very small singular values. Fortunately, inequality $\eqref{eq303}$ assures us that only the last singular value of $\bar{Y}$ has the possibility of being less than the tolerance. Therefore, it suffices to examine only the last singular value.
				\begin{itemize}
					
					\item[ (i)] If $\bar{\Sigma}(k+1, k+1) \geq \mathtt{tol}$, then
					$$
					Q \longleftarrow[Q \mid \widetilde{e}] \bar Q,\quad \Sigma \longleftarrow \bar\Sigma,\quad R \longleftarrow\left[\begin{array}{cc}
						R & 0 \\
						0 & 1
					\end{array}\right]\bar R.
					$$

					\item[(ii)] If $\bar{\Sigma}(k+1, k+1)<\mathtt{tol}$, then
					$$
					Q \longleftarrow[Q \mid \widetilde{e}] \bar Q(:, 1:k),\quad \Sigma \longleftarrow \bar\Sigma(1:k,1:k),\quad R \longleftarrow\left[\begin{array}{cc}
						R & 0 \\
						0 & 1
					\end{array}\right]\bar R(:,1:k).
					$$
				\end{itemize}
				It is essential to note that $p$-truncation and no-truncation  do not alter the previous data, whereas singular value truncation may potentially change the entire previous data. However, we can establish the following bound:
				\begin{lemma}\cite[Lemma 3.3]{ChenGang2023AISM}\label{error_SingularValueTruncation}
					Suppose $Q\Sigma R^\top$ to be the SVD of $A\in \mathbb R^{m\times n}$, where $\{\sigma_i\}_{i=1}^r$ are the positive singular values. Let $B = Q(:,1:r-1)\Sigma(1:r-1,1:r-1)(R(:,1:r-1))^\top$. We have:
					\begin{align*}
						\max\{|a_1 - b_1|, |a_2- b_2|, \ldots, |a_n-b_n|\} \le \sigma_r.
					\end{align*}
					Here, $a_i$ and $b_i$ correspond to the $i$-th columns present in matrices $A$ and $B$ respectively. The symbol $|\cdot|$ denotes the Euclidean norm within the realm of $\mathbb{R}^m$.
				\end{lemma}

\section{Incremental SVD compression for the Oldroyd model}\label{section 4}

We now integrate the incremental-SVD compression into the fully discrete Oldroyd solver. After computing a new velocity snapshot at each time step, we append that snapshot to the previously compressed history, update the SVD factors online, and use the compressed history in the assembly of the memory term at subsequent steps. The pressure is not stored in compressed form because it does not enter the convolution history term.

Relative to the standard fully discrete scheme, the only modification is that the history contribution is assembled from compressed velocity states. The compression routine consists of the $p$-truncation, no-truncation, and singular-value truncation steps described in Section~\ref{sec3}. The analysis uses two simple facts: $p$-truncation and no-truncation leave previously stored columns unchanged, whereas singular-value truncation perturbs them by at most a tolerance-dependent amount.

A convenient implementation can be summarized as follows.
\begin{enumerate}
\item Compute the first time-step solution with the standard scheme and compress the initial history $\{\bm u_h^0,\widehat{\bm u}_h^1\}$.
\item At time step $i+1$, assemble the memory term from the compressed history $\{\widetilde{\bm u}_h^{i,j}\}_{j=0}^{i}$, solve for $\widehat{\bm u}_h^{i+1}$, and update the SVD factors with the new snapshot.
\item Continue until the final time level $N$.
\end{enumerate}

With this notation, we seek $(\widehat {\bm{u}}_h^{n},\widehat{p}_h^{n}) \in \bm V_h\times Q_h$ such that

Based on the preceding discussion, we can present our formulation below, where we seek $(\widehat {\bm{u}}_h^{n},\widehat{p}_h^{n}) \in \bm V_h\times Q_h$ that satisfies the following equation:

\begin{align}\label{hat_equation-nonsingular}
	\begin{split}
	 &(d_t\widehat{\bm{u}}_h^n,\bm{v}_h)+\mathcal{A}(\bar{\widehat{\bm u}}_h^n,\bm{v}_h)+\mathcal{B}\!\left(\Delta t\sum_{j=1}^{n-1}K(\bar{t}_n-\bar{t}_j)\bar{\widetilde{\bm{u}}}_h^{n-1,j}+\frac{\Delta t}{2}K(0)\bar{\widehat{\bm{u}}}_h^n,\bm{v}_h\right)\\
	&\quad-\mathcal{D}(\bm{v}_h,\bar{\widehat{p}}_h^n)+\mathcal{C}(\bar{\widehat{\bm{u}}}_h^n,\bar{\widehat{\bm{u}}}_h^n,\bm{v}_h)=(\bm{f}(\bar{t}_n),\bm{v}_h),\\
	&\mathcal D(\bar{\widehat{\bm u}}_h^n,q_h)=0,\qquad \forall q_h\in Q_h.
	\end{split}
\end{align}

Here, $\left\{\widetilde{\bm u}_{i, j}\right\}_{j=0}^i$ represents the data that has been compressed from $\left\{\widetilde{\bm u}_{i-1,0}, \ldots\right.$, $\left.\tilde{\bm u}_{i-1, i-1}, \widehat{\bm u}_i\right\}$ using the incremental SVD algorithm, and 
\begin{align*}
	\bar{\widehat{\bm{u}}}_h^n=\frac{\widehat{\bm{u}}_h^n+\widehat{\bm{u}}_h^{n-1}}{2},\ 
	\bar{\widetilde{\bm{u}}}_h^{n-1,j}=\frac{\widetilde{\bm{u}}_h^{n-1,j}+\widetilde{\bm{u}}_h^{n-1,j-1}}{2}.
\end{align*} 
We assume that $Q_i, \Sigma_i, R_i$, and $W_i$ are the matrices associated with this compression process. In other words,
\begin{align*}
	\left[\widetilde{\bm u}_{i-1,0}|\cdots| \widetilde{\bm u}_{i-1, i-1} \mid \widehat{\bm u}_i\right] \stackrel{\text { Compress }}{\longrightarrow} Q_i\left[\Sigma_i R_i \mid W_i\right]^{\top}=\left[\widetilde{\bm u}_{i, 0}|\cdots| \widetilde{\bm u}_{i, i}\right].
\end{align*}
We need only update the matrix $ Q_i,\Sigma_i,R_i,W_i $ once we obtain new data from above process and consequently compress all available data into four new matrix, named still $ Q_i,\Sigma_i,R_i,W_i $.

We now examine the storage and work associated with the compressed history term. The factors $Q_i,\Sigma_i,R_i,W_i$ provide a history representation with storage cost $\mathcal O((m+n)r)$. The total work required to update the factors and assemble the history term over $n$ steps is
\begin{align*}
	\mathcal O(mnr)+\sum_{i=1}^{n}\sum_{j=1}^{i}\mathcal O(r)=\mathcal O(mnr+rn^2).
\end{align*}
Thus the compression reduces the dependence on the full spatial dimension from $m$ to the retained rank $r$, but it does \emph{not} remove the quadratic dependence on the number of time steps in direct history accumulation. This distinction is important for the theoretical positioning of the method.

\section{Error analysis for the compressed nonsingular-kernel scheme} \label{section 5}

In this section we estimate the error between the compressed scheme \eqref{hat_equation-nonsingular} and the exact solution of \eqref{OldryodModule-general}. The analysis is decomposed into three parts: the discretization error of the standard fully discrete scheme, the compression-induced perturbation between the standard and compressed schemes, and the resulting total error obtained by the triangle inequality.

For $\bm\psi\in\bm V$, let $\|\bm{\psi}\|_a^2=\mathcal{A}(\bm{\psi}, \bm{\psi})$. Then $\|\cdot\|_a$ defines a norm equivalent to $\|\cdot\|_1$ on $\bm V$ by \eqref{ass-A}. Throughout this section we assume that $\Omega$ is a bounded convex polyhedral domain and that the data of \eqref{OldryodModule-general} satisfy the following condition.

\begin{assumption}\label{assumption-1}
 Assume $\displaystyle K_0=\int_0^T K(t)\ {\rm d} t$ and $ K(t)\in H^2([0,T]) $,  and suppose  the following inequality to hold
\begin{align*}
		c_0K_0<1,
\end{align*}
					where $ c_0 $ is a constant such that 
					\begin{align*}
						|\mathcal{B}(\bm{u},\bm{v})|\le c_0\|\bm{u}\|_a\|\bm{v}\|_a,\ \forall \bm{u},\bm{v}\in \bm V.
					\end{align*}
\end{assumption}

				\begin{remark}
					We emphasize that \cref{assumption-1} ensures the dominance of the operator $ \mathscr{A} $ over the integral term. However, as elucidated further in  \Cref{singular kernel case}, our requirements necessitate solely $ K(t) \in  H^2
					[0, T] $ instead of assuming
					\cref{assumption-1}. This modified requirement holds if the kernel $ K(t) $ remains positive definite, and $ \mathscr{B}$ adheres to nonnegative symmetric properties. Our
					assumption, in this context, proves to be more encompassing and simpler to verify
					in comparison to the assumptions outlined in \cite{MR1686149}, where the prerequisites are more
					stringent to achieve a sharp decay rate.
				\end{remark}
			
Next, we deduce the error estimate between the solutions of \eqref{hat_equation-nonsingular} and \eqref{OldryodModule-general}. We divide this process into three parts. First, we establish the error bound between the solutions of \eqref{fully-discretization--nonsingular} and \eqref{OldryodModule-general}, then we deduce the error estimate between the solutions of \eqref{fully-discretization--nonsingular} and \eqref{hat_equation-nonsingular}. Finally, we apply triangle inequality to obtain our results.
\subsection{Convergence result of  fully discrete scheme \eqref{fully-discretization--nonsingular}}
In this subsection we record the stability and convergence properties of the standard fully discrete scheme \eqref{fully-discretization--nonsingular}. The result stated below is the baseline discretization estimate to which the compression error will later be added.
\begin{lemma}[Stability result]\label{stability--1}
Suppose that $ \bm{u}_h^n $ to be the solution of $ \eqref{fully-discretization--nonsingular} $  and  \Cref{assumption-1} to be satisfied, then it  holds 
	\begin{align*}
	\max_{1 \leq n \leq N}\|\bm{u}_h^n\|^2\le C\Delta t\sum_{n=1}^{N}\|\bm{f}(\bar{t}_n)\|^2+C\|\bm{u}_h^0\|^2.
	\end{align*}
\end{lemma}
\begin{proof}
	We notice that  solutions $ (\bm{u}_h^n,p_h^n) $ satisfy equations:
	\begin{align*}
		(d_t\bm{u}_h^{n},\bm{v}_h)+
		\mathcal{A}(\bar{\bm{u}}_h^{n},\bm{v}_h)
		+\mathcal{B}(\bar{\bm u}_{\rho,h}^{\Delta t,n},\bm{v}_h)
		-&\mathcal{D}(\bm{v}_h,\bar p_h^{n})
		+\mathcal{C}(\bar{\bm{u}}_h^n,\bar{\bm{u}}_h^{n},\bm{v}_h)=(\bm{f}(\bar t_n),\bm{v}_h),\\
		&\mathcal{D}(\bar{\bm{u}}_h^{n},q_h)=0,
	\end{align*}
	for all $ (\bm{v}_h,q_h )\in V_h\times Q_h $. Taking $ (\bm{v}_h,q_h)=(2\Delta t\bar{\bm{u}}_h^n,\bar p_h^n) $ in the above equations, it follows that
	\begin{align*}
		\|\bm{u}_h^n\|^2-\|\bm{u}_h^{n-1}\|^2+2\Delta t\|\bar{\bm{u}}_h^n\|_a^2
		=
		-2\Delta tB(\bar{\bm{u}}_{\rho,h}^{\Delta t,n},\bar{\bm{u}}_h^n)+2\Delta t(\bm{f}(\bar{t}_n),\bar{\bm{u}}_h^n).
	\end{align*}
	By Young inequality and \Cref{assumption-1}, it follows that 
	\begin{align*}
		&\|\bm{u}_h^n\|^2-\|\bm{u}_h^{n-1}\|^2+2\Delta t\|\bar{\bm{u}}_h^n\|_a^2\\
		&\le 2c_0\Delta t^2\sum_{j=1}^{n}K(\bar{t}_n-\bar{t}_j)\|\bar{\bm{u}}_h^j\|_a
		\|\bar{\bm{u}}_h^n\|_a+2\Delta t\|\bm{f}(\bar{t}_n)\|\|\bar{\bm{u}}_h^n\|\\
		&\le c_0\Delta t^2\sum_{j=1}^{n}K(\bar{t}_n-\bar{t}_j)(\|\bar{\bm{u}}_h^j\|_a^2+
		\|\bar{\bm{u}}_h^n\|_a^2)+C\Delta t\|\bm{f}(\bar{t}_n)\|^2+\varepsilon\Delta t\|\bar{\bm{u}}_h^n\|_a^2,
	\end{align*}
	for some $ \varepsilon\in (0,1) $, which be specified later. Then we    sum up with respect to n from $ 1 $ to  $ N $ and  obtain that
	\begin{align*}
		&\|\bm{u}_h^N\|^2-\|\bm{u}_h^0\|^2+2\Delta t\sum_{n=1}^{N}\|\bar{\bm{u}}_h^n\|_a^2\\
		&\le c_0\Delta t^2\sum_{n=1}^{N}
		\sum_{j=1}^{n}K(\bar{t}_n-\bar{t}_j)(\|\bar{\bm{u}}_h^j\|_a^2+
		\|\bar{\bm{u}}_h^n\|_a^2)+
		C\Delta t\sum_{n=1}^{N}\|\bm{f}(\bar{t}_n)\|^2+
		\varepsilon\Delta t\sum_{n=1}^{N}\|\bar{\bm{u}}_h^n\|_a^2\\
		&\le (2c_0K_0+\varepsilon)\Delta t\sum_{n=1}^{N}\|\bar{\bm{u}}_h^n\|_a^2
		+C\Delta t\sum_{n=1}^{N}\|\bm{f}(\bar{t}_n)\|^2.
	\end{align*}
	Since $ c_0K_0<1  $, we choose some $ \varepsilon\in (0,1) $ such that $ 2c_0K_0+\varepsilon<2 $ and apply the Gronwall's inequality, then   the equality becomes 
	\begin{align*}
		\max_{1 \leq n \leq N}\|\bm{u}_h^n\|^2\le C\Delta t\sum_{n=1}^{N}\|\bm{f}(\bar{t}_n)\|^2+C\|\bm{u}_h^0\|^2.
	\end{align*}
\end{proof}

\begin{lemma}[Convergence result]\label{error estimate-fem--1}

Suppose that $\bm u_h^n$ and $\bm u(t_n)$ are the solutions of \eqref{fully-discretization--nonsingular} and \eqref{OldryodModule-general}, respectively, and that \Cref{assumption-1} holds. Assume further that
\[
\bm u\in H^2(0,T;L^2(\Omega)^d)\cap L^\infty(0,T;H^{k+1}(\Omega)^d),\qquad
p\in L^\infty(0,T;H^{l+1}(\Omega)).
\]
Then, for all $1\le n\le N$,
\begin{align*}
	\|\bm{u}(t_n)-\bm{u}_h^n\|\le C\bigl(h^{\min\{k+1,l+1\}}+\Delta t^2\bigr).
\end{align*}

\end{lemma}

\begin{remark}

The estimate in \Cref{error estimate-fem--1} is the reference discretization bound used in the rest of the paper. For specific stable pairs, including the Mini element used in Section~\ref{section 7}, the observed spatial rate may be sharper than the generic statement above. The numerical results should therefore be interpreted as empirical validation rather than as evidence that the theorem is sharp in every regime.

\end{remark}

To  describe proof process better, we introduce throughout the following definition:
\begin{align*}
	&\bm{u}^n=\bm{u}(t_n),\  \bm{e}_n=\bm{u}_h^n-\bm{u}^n,\
	\bm{\theta}_n=\bm{u}_h^n-\bm P_h\bm{u}^n,\ \bm{\xi}_n=\bm P_h\bm{\bm{u}}^n-\bm{u}^n,\\
	&\bar{X}_K^{\Delta t,n}=\Delta t\sum_{j=1}^{n-1}K(\bar{t}_n-\bar{t}_j)\bar{X}_j+\frac{\Delta t}{2}K(0)\bar{X}_n,\ \text{where} \  X=\bm{u},\ \bm{\theta}, \ \text{or}\  \bm{\xi}.
\end{align*}
The following lemma is  frequently used in the proof of error estimation. 
\begin{lemma}
	\cite[Lemma 4]{GuoYingwen2022Ceaf}\label{integration by parts}
	Using the integration by parts, then for all $\varphi \in H^2([0,t])$, there hold
	\begin{align*}
		& \bar{\varphi}\left(t_n\right)-\frac{1}{\Delta t} \int_{t_{n-1}}^{t_n} \varphi(t)\ {\rm{d}}t=\frac{1}{2\Delta t} \int_{t_{n-1}}^{t_n}\left(t-t_{n-1}\right)\left(t_n-t\right) \varphi_{t t}(t)\ {\rm{d}}t \\
		& \bar{\varphi}\left(t_n\right)-\varphi\left(\bar{t}_n\right)=\frac{1}{2} \int_{t_{n-1}}^{\bar{t}_n}\left(t-t_{n-1}\right) \varphi_{t t}\ {\rm{d}}t+\frac{1}{2} \int_{\bar{t}_n}^{t_n}\left(t_n-t\right) \varphi_{t t}\ {\rm{d}}t, \\
		& \frac{\varphi\left(\bar{t}_n\right)}{2}-\frac{1}{\Delta t} \int_{\bar{t}_n}^{t_n} \varphi(t)\ {\rm{d}}t=-\frac{1}{\Delta t} \int_{\bar{t}_n}^{t_n}\left(t_n-t\right) \varphi_t(t)\ {\rm{d}}t.
	\end{align*}
\end{lemma}
Then we give detailed proof of \Cref{error estimate-fem--1}.
\begin{proof}
		We notice that $ \bm{u}(t) $ to satisfy equation \eqref{varitional-formula}, integrating \eqref{varitional-formula} from $ t_{n-1} $ to $ t_n $ to obtain 
		\begin{align}\label{integration--1}
			&(d_t\bm{u}^n,\bm{v}_h)+\frac{1}{\Delta t}\int_{t_{n-1}}^{t_n}\mathcal{A}(\bm{u}(t),\bm{v}_h)\ {\rm{d}}t+\mathcal{B}(\frac{1}{\Delta t}\int_{t_{n-1}}^{t_n}\int_0^tK(t-s)\bm{u}(s) {\rm{d}}s\ {\rm{d}}t,\bm{v}_h)\notag\\
			&+\frac{1}{\Delta t}\int_{t_{n-1}}^{t_n}\mathcal{C}(\bm{u}(t),\bm{u}(t),\bm{v}_h) {\rm{d}}t-\frac{1}{\Delta t}\int_{t_{n-1}}^{t_n}\mathcal{D}(\bm{v}_h,p(t))\ {\rm{d}}t=\frac{1}{\Delta t}\int_{t_{n-1}}^{t_n}(\bm{f}(t),\bm{v}_h)\ {\rm{d}}t.
		\end{align}
		Subtracting  \eqref{integration--1} from \eqref{fully-discretization--nonsingular}, it follows that 
		\begin{align}\label{error-equation(fem)--1}
			\begin{split}
				&(d_t \bm{e}_n,\bm{v}_h)+\mathcal{A}(\bar{\bm{e}}_n,\bm{v}_h)+\mathcal{B}(\bar{\bm{e}}_{K}^{\Delta t,n},\bm{v}_h)=\frac{1}{\Delta t}\int_{t_{n-1}}^{t_n}\mathcal{A}(\bm{u}(t),\bm{v}_h)\ {\rm{d}}t-\mathcal{A}(\bar{\bm{u}}^n,\bm{v}_h)\\
				&+\frac{1}{\Delta t}\mathcal{B}(\int_{t_{n-1}}^{t_n}\int_0^tK(t-s)\bm{u}(s)\ {\rm{d}}s\ {\rm{d}}t,\bm{v}_h)-\mathcal{B}(\bar{\bm{u}}_k^{\Delta t,n},\bm{v}_h)+\frac{1}{\Delta t}\int_{t_{n-1}}^{t_n}\mathcal{C}(\bm{u}(t),\bm{u}(t),\bm{v}_h)\ {\rm{d}}t\\
				&-\mathcal{C}(\bar{\bm{u}}_h^n,\bar{\bm{u}}_h^n,\bm{v}_h)\ {\rm{d}}t+\mathcal{D}(\bm{v}_h,\bar{p}_h^n-\frac{1}{\Delta t}\int_{t_{n-1}}^{t_n}p(t)\ {\rm{d}}t)+(\bm{f}(\bar{t}_n)-\frac{1}{\Delta t}\int_{t_{n-1}}^{t_n}\bm{f}(t)\ {\rm{d}}t,\bm{v}_h).
			\end{split}
		\end{align}
		Noting that $ \bm{e}_n=\bm{\theta}_n+\bm{\xi}_n $ and using the property of projection $ \bm{P}_h $, then \eqref{error-equation(fem)--1} becomes 
		\begin{align}\label{error-equation(fem)--2}
			\begin{split}
				&(d_t \bm{\theta}_n,\bm{v}_h)+\mathcal{A}(\bar{\bm{\theta}}_n,\bm{v}_h)+\mathcal{B}(\bar{\bm{\theta}}_{K}^{\Delta t,n},\bm{v}_h)
				=\frac{1}{\Delta t}\int_{t_{n-1}}^{t_n}\mathcal{A}(\bm{u}(t),\bm{v}_h)\ {\rm{d}}t-\mathcal{A}(\bar{\bm{u}}^n,\bm{v}_h)\\
				&+\frac{1}{\Delta t}\mathcal{B}(\int_{t_{n-1}}^{t_n}\int_0^tK(t-s)\bm{u}(s)\ {\rm{d}}s\ {\rm{d}}t,\bm{v}_h)-\mathcal{B}(\bar{\bm{u}}_k^{\Delta t,n},\bm{v}_h)+\frac{1}{\Delta t}\int_{t_{n-1}}^{t_n}\mathcal{C}(\bm{u}(t),\bm{u}(t),\bm{v}_h)\ {\rm{d}}t\\
				&-\mathcal{C}(\bar{\bm{u}}_h^n,\bar{\bm{u}}_h^n,\bm{v}_h)\ +\mathcal{D}(\bm{v}_h,\bar{p}_h^n-\frac{1}{\Delta t}\int_{t_{n-1}}^{t_n}p(t)\ {\rm{d}}t)+(\bm{f}(\bar{t}_n)-\frac{1}{\Delta t}\int_{t_{n-1}}^{t_n}\bm{f}(t)\ {\rm{d}}t,\bm{v}_h)\\
				&-\mathcal{A}(\bar{\bm{\xi}}_n,\bm{v}_h)-\mathcal{B}(\bar{\bm{\xi}}_K^{\Delta t,n},\bm{v}_h).
			\end{split}
		\end{align}
		We take $ \bm{v}_h=2\Delta t\bar{\bm{\theta}}_n $ in \eqref{error-equation(fem)--2} to arrive at 
		\begin{align*}
			&\|\bm{\theta}_n\|^2-\|\bm{\theta}_{n-1}\|^2+2\Delta t\|\bar{\bm{\theta}}_n\|_a^2+2\Delta t\mathcal{B}(\bar{\bm{\theta}}_K^{\Delta t,n},\bar{\bm{\theta}}_n)\\
			&=2\int_{t_{n-1}}^{t_n}\mathcal{A}(\bm{u}(t),\bar{\bm{\theta}}_n)\ {\rm{d}}t-2\Delta t\mathcal{A}(\bar{\bm{u}}^n,\bar{\bm{\theta}}_n)\\
			&+2\mathcal{B}(\int_{t_{n-1}}^{t_n}\int_{0}^{t}K(t-s)\bm{u}(s)\ {\rm{d}}s\ {\rm{d}}t,\bar{\bm{\theta}}_n)
			-2\Delta t\mathcal{B}(\bar{\bm{u}}_{K}^{\Delta t,n},\bar{\bm{\theta}}_n)\\
			&+2\int_{t_{n-1}}^{t_n}\mathcal{C}(\bm{u}(t),\bm{u}(t),\bar{\bm{\theta}}_n)\ {\rm{d}}t-2\Delta t\mathcal{C}(\bar{\bm{u}}_h^n,\bar{\bm{u}}_h^n,\bar{\bm{\theta}}_n)\\
			&+2\Delta t\mathcal{D}(\bar{\bm{\theta}}_n,\bar{p}_h^n-\frac{1}{\Delta t}\int_{t_{n-1}}^{t_n}p(t)\ {\rm{d}}t)
			+2\Delta t(\bm{f}(\bar t_n)-\frac{1}{\Delta t}\int_{t_{n-1}}^{t_n}\bm{f}(t)\ {\rm{d}}t,\bar{\bm{\theta}}_n)\\
			&-2\Delta t\mathcal{A}(\bar{\bm{\xi}}_n,\bar{\bm{\theta}}_n)-2\Delta t\mathcal{B}(\bar{\bm{\xi}}_K^{\Delta t,n},\bar{\bm{\theta}}_n)\\
			&=\sum_{i=1}^{10}R_i.
		\end{align*}
		Next, we turn to estimate the right  terms $\displaystyle \{R_i\}_{i=1}^{10} $. For the terms $ R_1+R_2 $, $R_8 $, we use \Cref{integration by parts}  and Young inequality to get
		\begin{align*}
			R_1+R_2&=2\Delta t\mathcal{A}(\frac{1}{\Delta t}\int_{t_{n-1}}^{t_n}\bm{u}(t)\ {\rm{d}}t-\bar{\bm{u}}^n,\bar{\bm{\theta}}_n)\\
			&=-\mathcal{A}(\int_{t_{n-1}}^{t_n}(t-t_{n-1})(t_n-t)\bm{u}_{tt}(t)\ {\rm{d}}t,\bar{\bm{\theta}}_n)\\
			&\le C\Delta t^{5/2}(\int_{t_{n-1}}^{t_n}\|\bm{u}_{tt}(t)\|_a^2\ {\rm{d}}t)^{1/2}\|\bar{\bm{\theta}}_n\|_a\\
			&\le C\Delta t^4\int_{t_{n-1}}^{t_n}\|\bm{u}_{tt}(t)\|_a^2\ {\rm{d}}t+\frac{\mu\Delta t}{7}\|\bar{\bm{\theta}}_n\|_a^2,\\
			R_8&=(\int_{t_{n-1}}^{t_n}(t-t_{n-1})(t_n-t)\bm{f}_{tt}(t)\ {\rm{d}}t,\bar{\bm{\theta}}_n)\\
			&-\Delta t(\int_{t_{n-1}}^{\bar{t}_n}(t-t_{n-1})\bm{f}_{tt}(t)\ {\rm{d}}t+\int_{\bar t_{n}}^{t_n}(t_n-t)\bm{f}_{tt}(t)\ {\rm{d}}t,\bar{\bm{\theta}}_n)\\
			&\le C\Delta t^{5/2}(\int_{t_{n-1}}^{t_n}\|\bm{f}_{tt}(t)\|^2\ {\rm{d}}t)^{1/2}\|\bar{\bm{\theta}}_n\|\\
			&\le C\Delta t^{4}\int_{t_{n-1}}^{t_n}\|\bm{f}_{tt}(t)\|^2\ {\rm{d}}t+\frac{\mu\Delta t}{7}\|\bar{\bm{\theta}}_n\|_a^2.	
		\end{align*}
	For $ R_7 $, noticing $ \bm{P}_h:\bm{V}\to \bm{W}_h $, one has
	\begin{align*}
		\mathcal{D}(\bar{\bm{\theta}}_n,q_h)=\mathcal{D}(\bar{\bm{u}}_h^n,q_h)-\mathcal{D}(\bm{P}_h\bar{\bm{u}}^n,q_h)=0, \quad \forall \ q_h\ \in Q_h.
\end{align*}
thus
\begin{align*}
		R_7&=2\Delta t\mathcal{D}(\bar{\bm{\theta}}_n,\bar{p}_h^n-\frac{1}{\Delta t}\int_{t_{n-1}}^{t_n}p(t)\ {\rm{d}}t)\\
	&=2\Delta t\mathcal{D}(\bar{\bm{\theta}}_n,\rho_h(\frac{1}{\Delta t}\int_{t_{n-1}}^{t_n}p(t)\ {\rm{d}}t)-\frac{1}{\Delta t}\int_{t_{n-1}}^{t_n}p(t)\ {\rm{d}}t)\\
	&\le Ch^{l+1}\Delta t\|\bar{\bm{\theta}}_n\|_a|\frac{1}{\Delta t}\int_{t_{n-1}}^{t_n}p(t)\ {\rm{d}}t|_{l+1}\\
	&\le \frac{\mu\Delta t}{7}\|\bar{\bm{\theta}}_n\|_a^2+Ch^{2l+2}\int_{t_{n-1}}^{t_n}|p(t)|_{l+1}^2\ {\rm{d}}t.
\end{align*}

		Interpolation approximation property and Young inequality applied to $ R_9 $ and $ R_{10} $ to obtain
		\begin{align*}
			R_9&\le C\Delta t\|\bar{\bm{\xi}}_n\|_a\|\bar{\bm{\theta}}_n\|_a
		\le C\Delta th^{2k}|\bm{u}^n|_{k+1}^2+\frac{\mu\Delta t}{7}\|\bar{\bm{\theta}}_n\|_a^2,\\
			R_{10}&=-2\Delta t\mathcal{B}(\Delta t\sum_{j=1}^{n-1}K(\bar{t}_n-\bar{t}_j)\bar{\bm{\xi}}_j+\frac{K(0)}{2}\Delta t\bar{\bm{\xi}}_n,\bar{\bm{\theta}}_n)\\
			&\le C\Delta t^2\sum_{j=1}^{n}K(\bar{t}_n-\bar{t}_j)\|\bar{\bm{\xi}}_j\|_a\|\bar{\bm{\theta}}_n\|_a\\
			&\le C\Delta th^{2k}\max_{1 \leq j \leq n}|\bm{u}^j|_{k+1}^2+\frac{\mu\Delta t}{7}\|\bar{\bm{\theta}}_n\|_a^2.
		\end{align*}
		As for $ R_3+R_4 $, we rewrite them as follows:
		\begin{align*}
			R_3+R_4&=2\Delta t\mathcal{B}(\frac{1}{\Delta t}\int_{t_{n-1}}^{t_n}\int_0^tK(t-s)\bm{u}(s)\ {\rm{d}}s\ {\rm{d}}t-\int_0^{\bar{t}_n}K(\bar{t}_n-s)\bm{u}(s)\ {\rm{d}}s,\bar{\bm{\theta}}_n)\\
			&+2\Delta t\sum_{j=1}^{n-1}\mathcal{B}(\int_{t_{j-1}}^{t_j}K(\bar{t}_n-s)\bm{u}(s)\ {\rm{d}}s-\Delta tK(\bar{t}_n-\bar{t}_j)\bm{u}(\bar{t}_j),\bar{\bm{\theta}}_n)\\
			&+2\Delta t\mathcal{B}(\int_{t_{n-1}}^{\bar{t}_n}K(\bar{t}_n-s)\bm{u}(s)\ {\rm{d}}s-\frac{\Delta t}{2}K(0)\bm{u}(\bar{t}_n),\bar{\bm{\theta}}_n)\\
			&+2\Delta t^2\sum_{j=1}^{n-1}K(\bar{t}_n-\bar{t}_j)\mathcal{B}(\bm u(\bar{t}_j)-\bar{\bm{u}}_j,\bar{\bm{\theta}}_n)+\Delta t^2K(0)\mathcal{B}(\bm{u}(\bar{t}_n)-\bar{\bm{u}}^n,\bar
			{\bm{\theta}}_n)\\
			&=\sum_{i=1}^{5}S_i.
		\end{align*}
		Then, we use the same way to estimate the terms $ \{S_i\}_{i=1}^5 $. Let 
		\begin{align*}
			F(t)=\int_0^tK(t-s)\bm{u}(s)\ {\rm{d}}s,\qquad\ G(t)=K(\bar{t}_n-t)\bm{u}(t),
		\end{align*}
		applying \Cref{integration by parts} and Young inequality, it follows that 
			\begin{align*}
				S_1&=-\mathcal{B}(\int_{t_{n-1}}^{t_n}(t-t_{n-1})(t_n-t)F_{tt}\ {\rm{d}}t,\bar{\bm{\theta}}_n)
				\le C\Delta t^{5/2}(\int_{t_{n-1}}^{t_n}\|F_{tt}\|_a^2\ {\rm{d}}t)^{1/2}\|\bar{\bm{\theta}}_n\|_a^2\\
				&\le C\Delta t^4\int_{t_{n-1}}^{t_n}\|F_{tt}\|_a^2\ {\rm{d}}t+\frac{\mu\Delta t}{35}\|\bar{\bm{\theta}}_n\|_a^2,\\
				S_2&=-\Delta t\sum_{j=1}^{n-1}\mathcal{B}(\int_{t_{j-1}}^{t_j}(t-t_{j-1})(t_j-t)G_{tt}\ {\rm{d}}t,\bar{\bm{\theta}}_n)
				\le C\Delta t\sum_{j=1}^{n-1}\Delta t^{5/2}(\int_{t_{j-1}}^{t_j}\|G_{tt}\|_a^2\ {\rm{d}}t)^{1/2}\|\bar{\bm{\theta}}_n\|_a\\
				&\le C\Delta t^{5}\int_{0}^{t_{n-1}}\|G_{tt}\|_a^2\ {\rm{d}}t +\frac{\mu\Delta t}{35}\|\bar{\bm{\theta}}_n\|_a^2,
			\end{align*}
		and
			\begin{align*}
				S_3&=-2\Delta t\mathcal{B}(\int_{t_{n-1}}^{\bar{t}_n}(t-t_{n-1})G_t\ {\rm{d}}t,\bar{\bm{\theta}}_n)
				\le C\Delta t^{5/2}(\int_{t_{n-1}}^{t_n}\|G_t\|_a^2\ {\rm{d}}t)^{1/2}\|\bar{\bm{\theta}}_n\|_a\\
				&\le C\Delta t^4\int_{t_{n-1}}^{t_n}\|G_t\|_a^2\ {\rm{d}}t+\frac{\mu\Delta t}{35}\|\bar{\bm{\theta}}_n\|_a^2,\\
				S_4&=-\Delta t^2\sum_{j=1}^{n-1}K(\bar{t}_n-\bar{t}_j)\mathcal{B}(\int_{t_{j-1}}^{\bar{t}_j}(t-t_{j-1})\bm{u}_{tt}\ {\rm{d}}t+\int_{\bar{t}_j}^{t_j}(t_j-t)\bm{u}_{tt}\ {\rm{d}}t,\bar{\bm{\theta}}_n)\\
				&\le \Delta t\sum_{j=1}^{n-1}K(\bar{t}_n-\bar t_j)\Delta t^{5/2}(\int_{t_{j-1}}^{t_j}\|\bm{u}_{tt}\|_a^2\ {\rm{d}}t)^{1/2}\|\bar{\bm{\theta}}_n\|_a\\
				&\le C\Delta t^5\int_0^{t_{n-1}}\|\bm{u}_{tt}\|_a^2\ {\rm{d}}t+\frac{\mu\Delta t}{35}\|\bar{\bm{\theta}}_n\|_a^2,\\
				S_5&=-\frac{\Delta t^2}{2}K(0)\mathcal{B}(\int_{t_{n-1}}^{\bar{t}_n}(t-t_{n-1})\bm{u}_{tt}\ {\rm{d}}t+\int_{\bar{t}_n}^{t_n}(t_n-t)\bm{u}_{tt}\ {\rm{d}}t,\bar{\bm{\theta}}_n)\\
				&\le C\Delta tK(0)\Delta t^{5/2}(\int_{t_{n-1}}^{t_n}\|\bm{u}_{tt}\|_a^2\ {\rm{d}}t)^{1/2}\|\bar{\bm{\theta}}_n\|_a\\
				&\le C\Delta t^4\int_{t_{n-1}}^{t_n}\|\bm{u}_{tt}\|_a^2\ {\rm{d}}t+\frac{\mu\Delta t}{35}\|\bar{\bm{\theta}}_n\|_a^2.
			\end{align*}
		Combining with  all above estimate for $S_1$ to $S_5 $, we conclude 
		\begin{align*}
			R_3+R_4&\le C\Delta t^4\int_{t_{n-1}}^{t_n}(\|F_{tt}\|_a^2+\|G_t\|_a^2+\|\bm{u}_{tt}\|_a^2)\ {\rm{d}}t\\
			&+C\Delta t^5\int_{0}^{t_{n-1}}(\|G_{tt}\|_a^2+\|\bm{u}_{tt}\|_a^2\ {\rm{d}}t)
			+\frac{\mu\Delta t}{7}\|\bar{\bm{\theta}}_n\|_a^2.
		\end{align*}
		Finally, we estimate $ R_5+R_6 $. We split it into several terms:
		\begin{align*}
			R_5+R_6&=(2\int_{t_{n-1}}^{t_n}\mathcal{C}(\bm{u}(t),\bm{u}(t),\bar{\bm{\theta}}_n)\ {\rm{d}}t-\Delta t\mathcal{C}(\bm u^n,\bm{u}^n,\bar{\bm{\theta}}_n)-\Delta t\mathcal{C}(\bm{u}^{n-1},\bm{u}^{n-1},\bar{\bm{\theta}}_n))\\
			&\quad+(\Delta t\mathcal{C}(\bm u^n,\bm{u}^n,\bar{\bm{\theta}}_n)+\Delta t\mathcal{C}(\bm{u}^{n-1},\bm{u}^{n-1},\bar{\bm{\theta}}_n)-2\Delta t\mathcal{C}(\bar{\bm u}^n,\bar{\bm{u}}^n,\bar{\bm{\theta}}_n))\\
			&\quad+(2\Delta t\mathcal{C}(\bar{\bm u}^n,\bar{\bm{u}}^n,\bar{\bm{\theta}}_n)-2\Delta t\mathcal{C}(\bar{\bm{u}}_h^n,\bar{\bm{u}}_h^n,\bar{\bm{\theta}}_n))\\
			&=T_1+T_2+T_3.
		\end{align*}
		Integrating by parts, we obtain 
		\begin{align*}
			T_1&=-\int_{t_{n-1}}^{t_n}(t-t_{n-1})(t_n-t)\mathcal{C}_{tt}(\bm{u}(t),\bm{u}(t),\bar{\bm{\theta}}_n)\ {\rm{d}}t\\
			&\le C\Delta t^{5/2}(\int_{t_{n-1}}^{t_n}\|\bm{u}_{tt}\|_0^2\|\bm{u}\|_2^2+\|\bm{u}_t\|_2^2\|\bm{u}_t\|_0^2\ {\rm{d}}t)^{1/2}\|\nabla\bar{\bm{\theta}}_n\|\\
			&\le \frac{\mu\Delta t}{21}\|\bar{\bm{\theta}}_n\|_a^2+C\Delta t^4\int_{t_{n-1}}^{t_n}(\|\bm{u}_{tt}\|_0^2\|\bm{u}\|_2^2+\|\bm{u}_t\|_2^2\|\bm{u}_t\|_0^2)\ {\rm{d}}t.
		\end{align*}
		By simple computation, we use the property of $ C(\cdot,\cdot,\cdot) $  and estimate $ T_2 $ as follows:
			\begin{align*}
				T_2&=\Delta t\mathcal{C}(\bm u^n,\bm{u}^n,\bar{\bm{\theta}}_n)+\Delta t\mathcal{C}(\bm{u}^{n-1},\bm{u}^{n-1},\bar{\bm{\theta}}_n)-2\Delta t\mathcal{C}(\bar{\bm u}^n,\bar{\bm{u}}^n,\bar{\bm{\theta}}_n)\\
				&=\frac{\Delta t}{2}\mathcal{C}(\bm{u}^n-\bm{u}^{n-1},\bm{u}^n-\bm{u}^{n-1},\bar{\bm{\theta}}_n)
				=\frac{\Delta t^3}{2}\mathcal{C}(d_t\bm{u}^n,d_t\bm{u}^n,\bar{\bm{\theta}}_n)\\
				&\le C\Delta t^3\|\nabla(d_t\bm{u}^n)\|_0^2\|\bar{\bm{\theta}}_n\|_a
				\le C\Delta t^5\|\nabla(d_t\bm{u}^n)\|_0^4+\frac{\mu\Delta t}{21}\|\bar{\bm{\theta}}_n\|_a^2.
			\end{align*}
		For the term $ T_3 $, we notice
		\begin{align*}
			\bar{\bm{u}}^n&=\bar{\bm{u}}^n-\bm P_h\bar{\bm{u}}^n+\bm P_h\bar{\bm{u}}^n-\bar{\bm{u}}_h^n+\bar{\bm{u}}_h^n=\bar{\bm{u}}_h^n-\bar{\bm{\theta}}_n-\bar{\bm{\xi}}_n,
		\end{align*}
	then we can rewrite $ T_3 $ as follows:
	\begin{align*}
		T_3=&-2\Delta t\mathcal{C}(\bar{\bm{u}}_h^n,\bar{\bm{\xi}}_n,\bar{\bm{\theta}}_n)
		-2\Delta t\mathcal{C}(\bar{\bm{\theta}}_n,\bar{\bm{u}}_h^n,\bar{\bm{\theta}}_n)
		+2\Delta t\mathcal{C}(\bar{\bm{\theta}}_n,\bar{\bm{\xi}}_n,\bar{\bm{\theta}}_n)\\
		&-2\Delta t\mathcal{C}(\bar{\bm{\xi}}_n,\bar{\bm{u}}_h^n,\bar{\bm{\theta}}_n)
		+2\Delta t\mathcal{C}(\bar{\bm{\xi}}_n,\bar{\bm{\xi}}_n,\bar{\bm{\theta}}_n)\\
		&=T_{31}+T_{32}+T_{33}+T_{34}+T_{35}.
	\end{align*}
By applying  \Cref{trilinear-property}  , it is obvious that
\begin{align*}
	T_{31}&\le C\Delta t\|\nabla\bar{\bm{u}}_h^n\|
	\|\nabla\bar{\bm{\xi}}_n\|\|\nabla\bar{\bm{\theta}}_n\|,\\
	T_{32}+T_{33}&\le C\Delta t\|\bar{\bm{\theta}}_n\|\|\nabla\bar{\bm{\theta}}_n\|(\|\nabla\bar{\bm{u}}_h^n\|+\|\nabla\bar{\bm{\xi}}_n\|)\\
	&+C\Delta t\|\nabla\bar{\bm{\theta}}_n\|^{3/2}\|\bar{\bm{\theta}}_n\|^{1/2}(\|\bar{\bm{u}}_h^n\|^{1/2}\|\nabla\bar{\bm{u}}_h^n\|^{1/2}+\|\bar{\bm{\xi}}_n\|^{1/2}\|\nabla\bar{\bm{\xi}}_n\|^{1/2})\\
	T_{34}&\le C\Delta t\|\nabla\bar{\bm{\xi}}_n\|\|\nabla\bar{\bm{u}}_h^n\|\|\nabla\bar{\bm{\theta}}_n\|\\
	T_{35}&\le C\Delta t\|\nabla\bar{\bm{\xi}}_n\|^2\|\nabla\bar{\bm{\theta}}_n\|.
\end{align*}
We use the stability result to conclude that
\begin{align*}
T_3&\le C\Delta th^{2k}|\bar{\bm{u}}^n|_{k+1}^2+\frac{\mu\Delta t}{21}\|\bar{\bm{\theta}}_n\|_a^2+C\Delta t\|\bar{\bm{\theta}}_n\|^2.
\end{align*}
Together with all estimates for $ T_1,T_2$, $ T_3 $, it follows that
\begin{align*}
	R_5+R_6&\le  C\Delta t^4\int_{t_{n-1}}^{t_n}(\|\bm{u}_{tt}\|_0^2\|\bm{u}\|_2^2+\|\bm{u}_t\|_2^2\|\bm{u}_t\|_0^2)\ {\rm{d}}t\\
	&+ C\Delta t^5\|\nabla(d_t\bm{u}^n)\|_0^4+
	 C\Delta th^{2k}|\bar{\bm{u}}^n|_{k+1}^2+\frac{\mu\Delta t}{7}\|\bar{\bm{\theta}}_n\|_a^2+C\Delta t\|\bar{\bm{\theta}}_n\|^2.
\end{align*}
		Combine all above estimates and sum on $ n $ from $ 1 $ to $ N $ to derive
		\begin{align}\label{error-equa--sum}
			\begin{split}
				&\|\bm{\theta}_N\|^2-\|\bm{\theta}_0\|^2+2\Delta t\sum_{n=1}^{N}\|\bar{\bm{\theta}}_n\|_a^2+2\Delta t\sum_{n=1}^{N}\mathcal{B}(\bar{\bm{\theta}}_K^{\Delta t,n},\bar{\bm{\theta}}_n)\\
				&\le C\Delta t^4\int_0^T(\|\bm{u}_{tt}\|_a^2+\|\bm{f}_{tt}\|_a^2+\|F_{tt}\|^2+\|G_{tt}\|_a^2
				+\|G_t\|_a^2+\|\bm{u}_{tt}\|^2\|\bm{u}\|_2^2+\|\bm{u}_t\|_2^2\|\bm{u}_t\|_0^2)\ {\rm{d}}t\\
				&+h^{2l+2}\int_{0}^{T}|p(t)|_{l+1}^2{\rm{d}}t+C\sum_{n=1}^{N}\Delta th^{2k}\max_{1\le n\le N}|\bm{u}^n|_{k+1}^2\\
				&+C\Delta t^5\sum_{n=1}^{N}\|\nabla(d_t\bm{u}^n)\|^4+C\Delta t\sum_{n=1}^{N}\|\bar{\bm{\theta}}_n\|^2+\mu\Delta t\sum_{n=1}^{N}\|\bar{\bm{\theta}}_n\|_a^2.
			\end{split}
		\end{align}
		Finally, we turn to bound the right hand term. From standard technique,  we can arrive at 
			\begin{align}\label{right-hand-term}
				\begin{split}
					&C\Delta t^5\sum_{n=1}^{N}\|d_t(\nabla\bm{u}^n)\|^4
					=C\Delta t\sum_{n=1}^{N}(\int_{\Omega}(\int_{t_{n-1}}^{t_n}\partial_t(\nabla\bm{u}){\rm{d}}t)^2\ {\rm d }\bm{x})^2 \\
					&\qquad\le C\Delta t^3\sum_{n=1}^{N}(\int_{\Omega}\int_{t_{n-1}}^{t_n}(\partial_t(\nabla\bm{u}))^2\ {\rm{d}}t\ {\rm{d}}\bm x)^2
					=C\Delta t^3\sum_{n=1}^{N}(\int_{t_{n-1}}^{t_n}\|\partial_t(\nabla\bm{u})\|^2\ {\rm{d}}t)^2\\
					&\qquad\le C\Delta t^4\int_{0}^{T}\|\partial_t(\nabla\bm{u})\|^4\ {\rm{d}}t.
				\end{split}
			\end{align}
		We use the \cref{assumption-1} and choose $ \mu\in(0,2-2c_0K_0) $ to reach 
		\begin{align}\label{A-B--positive}
			\begin{split}
				&2\Delta t\sum_{n=1}^{N}\|\bar{\bm{\theta}}_n\|_a^2+2\Delta t\sum_{n=1}^{N}\mathcal{B}(\bar{\bm{\theta}}_K^{\Delta t,n},\bar{\bm{\theta}}_n)\\
				&=	2\Delta t\sum_{n=1}^{N}\|\bar{\bm{\theta}}_n\|_a^2
				+2\Delta t\sum_{n=1}^{N}\mathcal{B}(\Delta t\sum_{j=1}^{n-1}K(\bar{t}_n-\bar{t}_j)\bar{\bm{\theta}}_j+\frac{\Delta t}{2}K(0)\bar{\bm{\theta}}_n,\bar{\bm{\theta}}_n)\\
				&\ge 2\Delta t\sum_{n=1}^{N}\|\bar{\bm{\theta}}_n\|_a^2-2\Delta t\sum_{n=1}^{N}c_0\sum_{j=1}^{n}K(\bar{t}_n-\bar{t}_j)\|\bar{\bm{\theta}}_j\|_a\|\bar{\bm{\theta}}_n\|_a\\
				&\ge 2\Delta t\sum_{n=1}^{N}\|\bar{\bm{\theta}}_n\|_a^2-\Delta t\sum_{n=1}^{N}c_0\sum_{j=1}^{n}K(\bar{t}_n-\bar{t}_j)(\|\bar{\bm{\theta}}_j\|_a^2+\|\bar{\bm{\theta}}_n\|_a^2)\\
				&\ge 2\Delta t\sum_{n=1}^{N}\|\bar{\bm{\theta}}_n\|_a^2-2c_0K_0\Delta t\sum_{n=1}^{N}\|\bar{\bm{\theta}}_n\|_a^2\\
				&\ge \mu\Delta t\sum_{n=1}^{N}\|\bar{\bm{\theta}}_n\|_a^2.
			\end{split}
		\end{align}
		Combining  \eqref{right-hand-term}, \eqref{A-B--positive}  and using the 
		Gronwall's inequality to \eqref{error-equa--sum} to give our result.

\end{proof}

		\subsection{Convergence result of scheme \eqref{hat_equation-nonsingular}}
		Next, we derive approximation result between the solutions of \eqref{fully-discretization--nonsingular} and \eqref{hat_equation-nonsingular}. Towards this end, we need the following several  lemmas.
		\begin{lemma} \label{hat-fem-error}
			Let $ \bm{u}_h^n $ and $ \widehat{\bm{u}}_h^n $ be the solutions of \eqref{fully-discretization--nonsingular} and \eqref{hat_equation-nonsingular}, respectively. Assume that  \Cref{assumption-1} is valid, then it holds:
			\begin{align}
				\|\bm{u}_h^{N}-\widehat{\bm{u}}_h^{N}\|\le C\sqrt{(1+\gamma^{-1})Te^T}\max_{1 \leq n \leq N}\max_{1 \leq j \leq n-1}\|\nabla(\widehat{\bm{u}}_h^j-\widetilde{\bm{u}}_{h}^{n-1,j})\|,
			\end{align}
			where $ C $ is a constant, independent of $ h $ and $ \Delta t $, $ \gamma\in (0,2c_0^{-1}K_0^{-1}-2) $.
		\end{lemma}
		
		\begin{proof}
			We notice that $ (\bm{u}_h^{n},p_h^{n}) $ and $ (\widehat{\bm{u}}_h^{n},\widehat p_h^{n}) $ satisfying  equations \eqref{fully-discretization--nonsingular} and \eqref{hat_equation-nonsingular}, introducing  the following notations:
			\begin{align*}
				\widehat{\bm{e}}_i=\bm{u}_h^i-\widehat{\bm{u}}_h^i ,\  \widetilde{\bm{e}}_{i,j}=\bm{u}_h^j-\widetilde{\bm{u}}_h^{i,j},
			\end{align*}  
			subtracting  \eqref{hat_equation-nonsingular} from \eqref{fully-discretization--nonsingular} to obtain
			\begin{align}\label{error_hat_tilde1}
				\begin{split}
					&(d_t\widehat{\bm e}_{n},\bm{v}_h)
					+\mathcal{A}(\bar{\widehat{\bm{e}}}_n,\bm{v}_h)+\mathcal{B}(\Delta t\sum_{j=1}^{n-1}K(\bar{t}_n-\bar{t}_j)\bar{\widetilde{\bm{e}}}_{n-1,j}+\frac{\Delta t}{2}K(0)\bar{\widehat{e}}_n,\bm{v}_h)\\
					&\quad-\mathcal{D}(\bm{v}_h,\bar{p}_h^n-\bar{\widehat{p}}_h^n)+\mathcal{C}(\bar{\bm{u}}_h^n,\bar{\bm{u}}_h^n,\bm{v}_h)-\mathcal{C}(\bar{\widehat{\bm{u}}}_h^n,\bar{\widehat{\bm{u}}}_h^n,\bm{v}_h)=0.
				\end{split}
			\end{align}
			Take $ \bm{v}_h=2\Delta t\bar{\widehat{\bm{e}}}_{n} $ in $\eqref{error_hat_tilde1}$ and sum on $ n $ from $ n=1 $ to $ N $  to give
			\begin{align*}
				&\|\widehat{\bm{e}}_{N}\|^2-\|\widehat{\bm{e}}_{0}\|^2+2\Delta t\sum_{n=1}^{N}\|\bar{\widehat{
						\bm{e}}}_n\|_a^2\\
				&=-2\Delta t\sum_{n=1}^{N}\mathcal{B}(\Delta t\sum_{j=1}^{n-1}K(\bar{t}_n-\bar{t}_j)\bar{\widetilde{\bm{e}}}_{n-1,j}+\frac{\Delta t}{2}K(0)\bar{\widehat{\bm{e}}}_n,\bar{\widehat{\bm{e}}}_n)\\
				\\
				&\quad+ 2\Delta t\sum_{n=1}^{N}\mathcal{C}(\bar{\widehat{\bm{u}}}_h^n,\bar{\widehat{\bm{u}}}_h^n,\bar{\widehat{\bm{e}}}_n)
				-2\Delta t\sum_{n=1}^{N}\mathcal{C}(\bar{\bm{u}}_h^n,\bar{\bm{u}}_h^n,\bar{\widehat{\bm e}}_n)\\
				&=R_1+R_2+R_3.
			\end{align*}
			Now we move to bound the terms $ \{R_i\}_{i=1}^{3} $. We use \Cref{assumption-1}, the triangle inequality and Young inequality    to arrive at 
			\begin{align*}
				R_1&\le 2c_0\Delta t^2\sum_{n=1}^{N}\sum_{j=1}^{n-1}K(\bar{t}_n-\bar{t}_j)\|\bar{\widetilde{\bm{e}}}_{n-1,j}\|_a\|\bar{\widehat{\bm{e}}}_n\|_a+c_0\Delta t^2\sum_{n=1}^{N}K(0)\|\bar{\widehat{\bm{e}}}_n\|_a^2\\
				&\le c_0\Delta t^2\sum_{n=1}^{N}\sum_{j=1}^{n-1}K(\bar{t}_n-\bar{t}_j)(\|\bar{\widetilde{\bm{e}}}_{n-1,j}\|_a^2+\|\bar{\widehat{\bm{e}}}_n\|_a^2)+c_0\Delta t^2\sum_{n=1}^{N}K(0)\|\bar{\widehat{\bm{e}}}_n\|_a^2\\
				& \le c_0\Delta t^2\sum_{n=1}^{N}\sum_{j=1}^{n-1}K(\bar{t}_n-\bar{t}_j)((\|\bar{\widetilde{\bm{e}}}_{n-1,j}-\bar{\widehat{\bm{e}}}_n\|_a+\|\bar{\widehat{\bm{e}}}_n\|_a)^2+\|\bar{\widehat{\bm{e}}}_n\|_a^2)
				+c_0\Delta t^2\sum_{n=1}^{N}K(0)\|\bar{\widehat{\bm{e}}}_n\|_a^2\\
				&\le c_0\Delta t^2\sum_{n=1}^{N}\sum_{j=1}^{n-1}K(\bar{t}_n-\bar{t}_j)((1+\gamma)\|\bar{\widehat{\bm{e}}}_n\|_a^2+(1+\gamma^{-1})\|\bar{\widetilde{\bm{e}}}_{n-1,j}-\bar{\widehat{\bm{e}}}_n\|_a^2)\\
				&\quad +c_0\Delta t^2\sum_{n=1}^{N}\sum_{j=1}^{n}K(\bar{t}_n-\bar{t}_j)\|\bar{\widehat{\bm{e}}}_n\|_a^2\\
				&\le (2+\gamma)c_0K_0\Delta t\sum_{n=1}^{N}\|\bar{\widehat{\bm{e}}}_n\|_a^2
				+(1+\gamma^{-1})c_0K_0\Delta t\sum_{n=1}^{N}\max_{1 \leq j \leq n-1}\|\bar{\widetilde{\bm{e}}}_{n-1,j}-\bar{\widehat{\bm{e}}}_n\|_a^2
			\end{align*}
			for some $ \gamma\in (0,1).$
			For the term $ R_2 $ and $ R_3 $, it follows from  property of the trilinear function $ C(\cdot,\cdot,\cdot) $ that
			\begin{align*}
				R_2+R_3&=-2\Delta t\sum_{n=1}^{N}\mathcal{C}(\bar{\bm{u}}_h^{n},\bar{\bm{u}}_h^{n},\bar{\widehat{\bm{e}}}_{n})+2\Delta t\sum_{n=1}^{N}\mathcal{C}(\bar{\widehat{\bm{u}}}_h^{n},\bar{\widehat{\bm{u}}}_h^{n},\bar{\widehat{\bm{e}}}_{n})\\
				&=-2\Delta t\sum_{n=1}^{N}\mathcal{C}(\bar{\bm{u}}_h^{n},\bar{\bm{u}}_h^{n},\bar{\widehat{\bm{e}}}_{n})+2\Delta t\sum_{n=1}^{N}\mathcal{C}(\bar{\widehat{\bm{u}}}_h^{n},\bar{\bm{u}}_h^{n},\bar{\widehat{\bm{e}}}_{n})\\
				&=-2\Delta t\sum_{n=1}^{N}\mathcal{C}(\bar{\widehat{\bm{e}}}_{n},\bar{\bm{u}}_h^{n},\bar{\widehat{\bm{e}}}_{n})\\
				&\le C\Delta t\sum_{n=1}^{N}(\|\bar{\bm{u}}_h^n\|^{1/2}\|\nabla\bar{\bm{u}}_h^n\|^{1/2}\|\nabla\bar{\widehat{\bm{e}}}_{n}\|+\|\bar{\widehat{\bm{e}}}_n\|^{1/2}\|\nabla\bar{\widehat{\bm{e}}}_n\|^{1/2}\|\bar{\bm{u}}_h^n\|)\|\bar{\widehat{\bm{e}}}_n\|^{1/2}\|\nabla\bar{\widehat{\bm{e}}}_n\|^{1/2}\\
				&\le C\Delta t\sum_{n=1}^{N}\|\bar{\widehat{\bm{e}}}_n\|^2+\mu\Delta t\sum_{n=1}^{N}\|\bar{\widehat{\bm{e}}}_n\|_a^2.
			\end{align*}
			Then, since $ c_0K_0<1 $, we choose $ \gamma\in (0,2c_0^{-1}K_0^{-1}-2), \mu \in (0,2-(2+\gamma)c_0K_0) $ such that
			\begin{align*}
				(2+\gamma)c_0K_0+\mu<2,
			\end{align*}
			and  use the above estimate  to give
			\begin{align}\label{hat_sum}
				\begin{split}
					\|\widehat{\bm e}_N\|^2&
					\le \|\widehat{\bm{e}}_0\|^2+(1+\gamma^{-1})c_0K_0\Delta t\sum_{n=1}^{N}\max_{1 \leq j \leq n-1}\|\bar{\widetilde{\bm{e}}}_{n-1,j}-\bar{\widehat{\bm{e}}}_n\|_a^2
					+C\Delta t\sum_{n=1}^{N}\|\bar{\widehat{\bm{e}}}_n\|^2\\
					&\le \|\widehat{\bm{e}}_0\|^2+
					(1+\gamma^{-1})T\max_{1 \leq n \leq N}\max_{1 \leq j \leq n-1}\|\bar{\widetilde{\bm{e}}}_{n-1,j}-\bar{\widehat{\bm{e}}}_j\|_a^2+C\Delta t\sum_{n=1}^{N}\|\bar{\widehat{\bm{e}}}_n\|^2.
				\end{split}
			\end{align}
The  proof is finished by applying $ \widehat{\bm{e}}_0=0 $  and the Gronwall's inequality  to \eqref{hat_sum}.
\end{proof}

\begin{lemma}\label{hat-tilde-error}

Assume that $\{\widehat{\bm{u}}_h^i\}_{i=1}^n$ is the solution sequence generated by \eqref{hat_equation-nonsingular}, and that $\{\widetilde{\bm{u}}_h^{i,j}\}_{j=0}^{i}$ denotes the compressed history obtained from $\{\widetilde{\bm{u}}_h^{i-1,0},\ldots,\widetilde{\bm{u}}_h^{i-1,i-1},\widehat{\bm{u}}_h^i\}$ by the incremental SVD algorithm with tolerance $\mathtt{tol}$ applied to both $p$-truncation and singular-value truncation. Let $T_{\rm sv}$ be the total number of singular-value truncations. Then
\begin{align*}
	\max _{1 \leq n \leq N} \max _{1 \leq j \leq n-1}\left\|\widetilde{\bm{u}}_h^{n-1, j}-\widehat{\bm{u}}_h^j\right\|_a \leq(T_{\rm sv}+1) \sqrt{\sigma(S)} \,\mathtt{tol},
\end{align*}
where $\sigma(S)$ denotes the spectral radius of the stiffness matrix
\begin{align*}
	S=(\sum_{k,l=1}^{d}a_{kl}(\bm x)\frac{\partial\bm{\phi}_j}{\partial x_k},\frac{\partial\bm{\phi}_i}{\partial x_l})+(a(\bm x)\bm{\phi}_j,\bm{\phi}_i).
\end{align*}

\end{lemma}

\begin{proof}
We let   $\widehat{\bm{u}}_i$ and $\widetilde{\bm{u}}_{k,l}$ be the coefficients of $\widehat{\bm{u}}_h^i$ and $\widetilde{\bm{u}}_h^{k, \ell}$ corresponding to the finite element basis functions $\left\{\bm{\phi}_s\right\}_{s=1}^{m}$, respectively, i.e.
\begin{align*}
\widehat{\bm{u}}_h^i=\sum_{j=1}^{m}(\widehat{\bm{u}}_i)_j\bm{\phi}_j,\ 
\widetilde{\bm{u}}_h^{k,\ell}=\sum_{j=1}^{m}(\widetilde{\bm{u}}_{k,\ell})_j\bm{\phi}_j,
\end{align*} 
where $ (\bm{\alpha})_j $ denote the $j$-th component of the vector $ \bm{\alpha} $.
Then we  can obtain  the following inequality for $1 \leq j \leq$ $n-1 \leq N-1$ :
\begin{align*}
		\left\|\widetilde{\bm{u}}_{h}^{n-1, j}-\widehat{\bm{u}}_{h}^j\right\|_a 
		& \leq \sum_{k=j+1}^{n-1}\left\|\widetilde{\bm{u}}_{h}^{k, j}-\widetilde{\bm{u}}_{h}^{k-1, j}\right\|_a+\left\|\widetilde{\bm{u}}_{h}^{j, j}-\widehat{\bm{u}}_{h}^j\right\|_a \\
		& =\sum_{k=j+1}^{n-1} \sqrt{\left(\widetilde{\bm{u}}_{k, j}-\widetilde{\bm{u}}_{k-1, j}\right)^{\top}  A\left(\widetilde{\bm{u}}_{k, j}-\widetilde{\bm{u}}_{k-1, j}\right)}\\
		&\quad+\sqrt{\left(\widetilde{\bm{u}}_{j, j}-\widehat{\bm{u}}_j\right)^{\top} A\left(\widetilde{\bm{u}}_{j, j}-\widehat{\bm{u}}_j\right)} \\
		& \leq \sum_{k=j+1}^{n-1} \sqrt{\sigma(S)}\left|\widetilde{\bm{u}}_{k, j}-\widetilde{\bm{u}}_{k-1, j}\right|+\sqrt{\sigma(S)}\left|\widetilde{\bm{u}}_{j, j}-\widehat{\bm{u}}_j\right|,
\end{align*}
where 
\begin{align*}
	A=(\sum_{k,l=1}^{d}a_{kl}(\bm x)\frac{\partial\bm{\phi}_j}{\partial x_k},\frac{\partial\bm{\phi}_i}{\partial x_l})+(a(\bm x)\bm{\phi}_j,\bm{\phi}_i),
\end{align*}
and $\sigma(S)$ is the spectral radius of the stiffness matrix $S$. It is notable that $\widetilde{\bm{u}}_{k, j}$ corresponds to the $j$-th compressed data at the $k$-th step for  velocity, as illustrated by:
	\begin{align*}
	\left[\widetilde{\bm{u}}_{k-1,0}\left|\widetilde{\bm{u}}_{k-1,1}\right| \ldots\left|\widetilde{\bm{u}}_{k-1, k-2}\right| \widehat{\bm{u}}_{k-1}\right] \stackrel{\text { Compressed }}{\longrightarrow}\left[\widetilde{\bm{u}}_{k, 0}\left|\widetilde{\bm{u}}_{k, 1}\right| \ldots \mid \widetilde{\bm{u}}_{k, k}\right].
\end{align*}
Furthermore, considering that both $p$ truncation and no truncation maintain the prior data unchanged, it follows that at most $\min \left\{T_{\rm{sv}}, n-1-j\right\}$ terms of $\left\{\widetilde{\bm{u}}_{k, j}-\widetilde{\bm{u}}_{k-1, j}\right\}_{k=j+1}^{n-1}$ are non-zero, where $T_{\rm{sv}}$ represents the total number of times singular value truncation is applied. Consequently, for any $1 \leq j \leq n-1 \leq N-1$, we can derive:
\begin{align}\label{tol_A1}
	\max _{1 \leq n\leq N} \max _{1 \leq j \leq n-1}\|\widetilde{\bm{u}}_{h}^{n, j}-\widehat{\bm{u}}_{h}^j\|_a \leq\left(T_{sv}+1\right) \sqrt{\sigma(S)}\,\mathtt{tol},
\end{align}
so we prove our conclusion.
\end{proof}

Combining \Cref{error estimate-fem--1,hat-fem-error,hat-tilde-error}, we apply triangle inequality to arrive at our final conclusion:
		
\begin{theorem}\label{main_result}

Let $(\bm{u}(t),p(t))$ and $(\widehat{\bm{u}}_h^n,\widehat{p}_h^n)$ be the solutions of \eqref{OldryodModule-general} and \eqref{hat_equation-nonsingular}, respectively. Assume that the tolerance $\mathtt{tol}$ is used in both $p$-truncation and singular-value truncation throughout the incremental SVD procedure. Under \Cref{assumption-1}, if
\[
\bm u\in H^2(0,T;L^2(\Omega)^d)\cap L^\infty(0,T;H^{k+1}(\Omega)^d),\qquad
p\in L^\infty(0,T;H^{l+1}(\Omega)),
\]
then there exists a constant $C_T>0$, independent of $h$, $\Delta t$, and $\mathtt{tol}$, such that
\begin{align*}
	\|\bm{u}(t_n)-\widehat{\bm{u}}_h^n\|
	\le C_T\Bigl(h^{\min\{k+1,l+1\}}+\Delta t^2+(1+T_{\rm sv})\sqrt{\sigma(S)}\,\mathtt{tol}\Bigr).
\end{align*}

\end{theorem}

\begin{remark}

The theorem gives an explicit tolerance-dependent perturbation bound. Therefore, statements such as ``near machine precision'' should be understood as empirical observations for a particular tolerance choice (for example $\mathtt{tol}=10^{-12}$ in Section~\ref{section 7}), not as mesh-independent theorem-level guarantees.

\end{remark}

\section{Extension to weakly singular kernels}\label{singular kernel case}\label{section 6}

In this section we extend the compression strategy to the case in which the memory kernel is weakly singular. The principal difference from the nonsingular setting lies in the time discretization of the convolution term: instead of the midpoint rule used in Section~\ref{section 2}, we employ convolution quadrature adapted to the tempered weakly singular kernel. The compression mechanism itself is unchanged.

We consider the problem

		\begin{align}\label{OldryodModule-general-singular}
			\begin{split}
				\bm{u}_t+\mathscr{A} \bm{u}+\int_0^t K(t-s)\mathscr{B} \bm{u}(s){\rm d}s+(\bm{u}\cdot \nabla)\bm{u}+\nabla p&=\bm{f},\quad \text{in } \Omega\times \left(  0,T \right], \\
				\nabla \cdot\bm{u}&=0, \quad \text{in } \Omega\times \left(  0,T \right],\\
				\bm{u}&=\bm 0,\quad \text{on } \partial \Omega\times \left(  0,T \right],\\
				\bm{u}(\bm x,0)&=\bm{u}_0(\bm x),\quad \text{in } \Omega,
			\end{split}
		\end{align}

with tempered weakly singular kernel

		\begin{align*}
			K(t)=e^{-\lambda t}\frac{1}{\Gamma(\alpha)}t^{\alpha-1},\qquad \alpha\in(0,1).
		\end{align*}

The case $\lambda=0$ reduces to the Abel kernel.

		\subsection{Fully discrete scheme}

		To derive the fully discrete approximation of \eqref{OldryodModule-general-singular}, we use the trapezoidal convolution-quadrature rule from \cite{MR4626456}:

\begin{align}\label{weights}
			\begin{split}
				\omega_n^{(\alpha,\lambda)}&=e^{-\lambda t_n}\omega_n^{(\alpha,0)},\\
				\omega_n^{(\alpha,0)}&=2^{-\alpha}\sum_{s=0}^{n}\sigma_s^{(\alpha)}\beta_{n-s}^{(\alpha)},\\
				\rho_n^{(\alpha,\lambda)}&=\frac{e^{-\lambda t_n}t_n^{\alpha}}{\Gamma(\alpha+1)}-\Delta t^{\alpha}\sum_{p=0}^{n}e^{-\lambda(t_n-t_p)}\omega_p^{(\alpha,\lambda)},\\
				\sigma_s^{(\alpha)}&=\frac{\Gamma(\alpha+s)}{\Gamma(\alpha)\Gamma(s+1)}=\frac{\alpha(\alpha+1)\cdots(\alpha+s-1)}{s!},\\
				\beta_s^{(\alpha)}&=\frac{\Gamma(\alpha+1)}{\Gamma(\alpha-s+1)\Gamma(s+1)}=\frac{\alpha(\alpha-1)\cdots(\alpha-s+1)}{s!}.
			\end{split}
		\end{align}

		and corresponding numerical quadrature :
		\begin{align*}
			Q_n^{(\alpha,\lambda)}(\bm{u})&=\Delta t^{\alpha}\sum_{p=0}^{n}\omega_p^{(\alpha,\lambda)}\bm
			u(t_{n-p})+\rho_n^{(\alpha,\lambda)}\bm{u}(0),\\
			Q_n^{(\alpha,\lambda)}(\bm{u}_h)&=\Delta t^{\alpha}\sum_{p=0}^{n}\omega_p^{(\alpha,\lambda)}\bm
			{u}_h^{n-p}+\rho_n^{(\alpha,\lambda)}\bm{u}_h^0.
		\end{align*}
		Similarly, we can derive the standard Crank-Nicolson discretization scheme \eqref{fully-discretization--singular} and compressed method comprising the incremental SVD method \eqref{hat-euqation-singular} as follows:
		
		\begin{align}\label{fully-discretization--singular}
			\begin{split}
				&(d_t\bm{u}_h^n,\bm{v}_h)+\mathcal A(\bar{\bm{u}}_h^n,\bm{v}_h)+\mathcal B\!\left(\Delta t^{\alpha}\sum_{p=0}^{n}\omega_{p}^{(\alpha,\lambda)}\bar{\bm{u}}_h^{n-p}+\bar{\rho}_n^{(\alpha,\lambda)}\bm{u}_h^0,\bm{v}_h\right)\\
				&\quad-\mathcal D(\bm{v}_h,\bar{p}_h^n)+\mathcal C(\bar{\bm{u}}_h^n,\bar{\bm{u}}_h^n,\bm{v}_h)=(\bm{f}(\bar{t}_n),\bm{v}_h),\\
				&\mathcal D(\bar{\bm u}_h^n,q_h)=0,\qquad \forall q_h\in Q_h,
			\end{split}
		\end{align}
		\begin{align}\label{hat-euqation-singular}
			\begin{split}
				(d_t\widehat{\bm{u}}_h^n,\bm{v}_h)
				&+\mathcal A(\bar{\widehat{\bm{u}}}_h^n,\bm{v}_h)
				+\mathcal B\!\left(\Delta t^{\alpha}\sum_{p=1}^{n}\omega_{p}^{(\alpha,\lambda)}\bar{\widetilde{\bm{u}}}_h^{n-1,n-p}+\Delta t^{\alpha}\omega_0^{(\alpha,\lambda)}\bar{\widehat{\bm{u}}}_h^n+\bar{\rho}_n^{(\alpha,\lambda)}\widehat{\bm{u}}_h^0,\bm{v}_h\right)\\
				&-\mathcal D(\bm{v}_h,\bar{\widehat{p}}_h^n)+\mathcal C(\bar{\widehat{\bm{u}}}_h^n,\bar{\widehat{\bm{u}}}_h^n,\bm{v}_h)=(\bm{f}(\bar{t}_n),\bm{v}_h),\\
				&\mathcal D(\bar{\widehat{\bm u}}_h^n,q_h)=0,\qquad \forall q_h\in Q_h,
			\end{split}	
		\end{align}
		where 
		\begin{align*}
			\bar{\bm{u}}_h^i=\frac{\bm{u}_h^i+\bm{u}_h^{i-1}}{2},\
			\bar{\rho}_n^{(\alpha,\lambda)}=\frac{\rho_n^{(\alpha,\lambda)}+\rho_{n-1}^{(\alpha,\lambda)}}{2},\ \bm{u}_h^{-1}=\bm{0},
		\end{align*}
		and $\bm{u}_h^0=\bm P_h\bm u_0$, $\widehat{\bm u}_h^0=\bm P_h\bm u_0$.

Note that 
		\begin{align*}
			&\frac{1}{2}Q_n^{(\alpha,\lambda)}(\bm{u}_h)
			+\frac{1}{2}Q_{n-1}^{(\alpha,\lambda)}(\bm{u}_h)\\
			&\quad=\frac{1}{2}\Delta t^{\alpha}\sum_{p=0}^{n}\omega_p^{(\alpha,\lambda)}\bm{u}_h^{n-p}
			+\frac{1}{2}\rho_{n}^{(\alpha,\lambda)}\bm{u}_h^0
			+\frac{1}{2}\Delta t^{\alpha}\sum_{p=0}^{n-1}\omega_p^{(\alpha,\lambda)}\bm{u}_h^{n-1-p}
			+\frac{1}{2}\rho_{n-1}^{(\alpha,\lambda)}\bm{u}_h^0\\
			&\quad=\Delta t^{\alpha}\sum_{p=0}^{n}\omega_{p}^{(\alpha,\lambda)}\bar{\bm{u}}_h^{n-p}+\bar{\rho}_n^{(\alpha,\lambda)}\bm{u}_h^0, 
		\end{align*}
		and the weights sequences $ \omega_{p}^{(\alpha,\lambda)} $ defined in \eqref{weights} satisfy the following positive definite property:
		\begin{lemma}\label{positive-define property}
			Let $ \omega_{p}^{(\alpha,\lambda)} $ be defined in \eqref{weights}, then for any positive integer $ N $ and $ \widetilde{\bm{v}}_1,\widetilde{\bm{v}}_2,\cdots,\widetilde{\bm{v}}_N\in (H_0^1(\Omega))^d $, the following inequality holds:
			\begin{align}\label{sequence-positive}
				\sum_{n=1}^{N}\sum_{p=0}^{n}\omega_p^{(\alpha,\lambda)}(\widetilde{\bm v}_{n-p},\widetilde{\bm v}_n)\ge 0.
			\end{align}
			If $ \mathcal B $ is a symmetric non-negative definite elliptic operator, then 
			\begin{align}\label{oprator-sequence-positive}
				\sum_{n=1}^{N}\sum_{p=0}^{n}\omega_p^{(\alpha,\lambda)}\mathcal B(\widetilde{\bm v}_{n-p},\widetilde{\bm v}_n)\ge 0.
			\end{align}
		\end{lemma}
		\begin{proof}
			The equality \eqref{sequence-positive} is a straight result of \cite[Lemma 5]{MR4626456}. Combining \cite[lemma 4.2]{MR1686149} with \eqref{sequence-positive}, \eqref{oprator-sequence-positive} can be obtained immediately.
		\end{proof}
	
	\subsection{Error analysis of scheme \eqref{fully-discretization--singular}}
	
	As in Section~\ref{section 5}, we split the analysis into three parts: the error of the standard fully discrete scheme, the perturbation introduced by history compression, and the total error obtained by combining the two. The weakly singular kernel requires one additional structural assumption on the bilinear form $\mathcal B$ and on the convolution weights.

	\begin{assumption}\label{assumption-singular}
	
	\begin{enumerate}
		\item[(A1)] The bilinear form $\mathcal B(\cdot,\cdot)$ is symmetric and nonnegative on $\bm V$.
		\item[(A2)] The weight sequence $\{\omega_p^{(\alpha,\lambda)}\}_{p\ge 0}$ defined in \eqref{weights} satisfies the positive-definite property \eqref{oprator-sequence-positive}.
	\end{enumerate}
	
	\end{assumption}

	Under these assumptions, the following stability result and quadrature error estimate are the key ingredients in the weakly singular analysis.

\begin{lemma}[Stability result]
 			\label{stability--2}
			Suppose that $\bm{u}_h^n$ is the solution of \eqref{fully-discretization--singular} and that \Cref{assumption-singular} holds. Then, for all $n\ge 1$,
			\begin{align*}
				\max_{1 \leq n \leq N}\|\bm{u}_h^n\|^2\le C\Delta t\sum_{n=1}^{N}\|\bm{f}(\bar{t}_n)\|^2+C\|\nabla\bm{u}_h^0\|^2.
			\end{align*}
		\end{lemma}
		The proof of \cref{stability--2} is similar to \cref{stability--1}, so we omit it here.

		\begin{lemma}{\cite[Lemma 4]{MR4626456}}
			\label{quadrature-error}
			If $\varphi \in C^1([0, T])$, $\varphi^{\prime \prime}$ is continuous and integrable on $(0, T]$, then, for $n \geq 1, \alpha \in(0,1)$ and $\lambda \in[0, \infty)$,
			\begin{align*}
				\left|\breve{\varepsilon}^{(\alpha, \lambda)}(\varphi)\left(t_n\right)\right| \leq & C\left[\Delta t^2 t_n^{\alpha-1}\left|\varphi^{\prime}(0)\right|+\Delta t^{\alpha+1} \int_{t_{n-1}}^{t_n}\left|(p \varphi)^{\prime \prime}(s)\right| {\rm d} s\right. \\
				& \left.+\Delta t^2 \int_0^{t_{n-1}}\left(t_n-s\right)^{\alpha-1}\left|(p \varphi)^{\prime \prime}(s)\right| {\rm d} s\right], \quad n=1,2, \ldots, N,
			\end{align*}
			where $p(t)=e^{\lambda t}$ and $\breve{\varepsilon}^{(\alpha, \lambda)}(\varphi)\left(t_n\right)  $ denote the convolution quadrature error:
			\begin{align*}
				\breve{\varepsilon}^{(\alpha, \lambda)}(\varphi)\left(t_n\right)=\int_{
					0}^{t_n}K(t_n-s)\phi(s)\ {\rm d}s-Q_n^{(\alpha,\lambda)}(\phi).
			\end{align*}
		\end{lemma}
	\begin{remark}
		From \cref{quadrature-error}, it follows that if $\left|\varphi^{\prime}(0)\right|=0$ and $\left|\varphi^{\prime \prime}(t)\right| \leq C$ for $t \in[0, T]$, then the quadrature error is $\mathcal{O}(\Delta t^2)$.
	\end{remark}

		Based on the above lemmas, we can establish the following error estimate between the solutions of \eqref{fully-discretization--singular} and \eqref{OldryodModule-general-singular}.
		
		\begin{lemma}\label{error estimate-fem--2}
		
		Let $\bm u_h^n\in\bm V_h$ and $\bm u(t_n)\in \bm V$ be the solutions of \eqref{fully-discretization--singular} and \eqref{OldryodModule-general-singular}, respectively. Assume that \Cref{assumption-singular} holds and that
		\[
		\bm u\in L^\infty(0,T;H^{k+1}(\Omega)^d),\qquad p\in L^\infty(0,T;H^{l+1}(\Omega)),
		\]
		and that $\bm u$ satisfies the regularity condition
		\begin{align}\label{regularity-condition}
			\left\|\bm{u}_t\right\| \leq C, \quad\left\|\bm{u}_{tt}\right\| \leq C e^{-\lambda t} t^{\alpha-1}, \quad\left\|\bm{u}_{ttt}\right\| \leq C e^{-\lambda t} t^{\alpha-2}, \quad t \rightarrow 0^{+},
		\end{align}
		where $C$ is independent of $h$ and $\Delta t$. Then, for all $1\le n\le N$,
		\begin{align*}
			\|\bm{u}(t_n)-\bm{u}_h^n\|\le C\bigl(h^{\min\{k+1,l+1\}}+\Delta t^{1+\alpha}\bigr).
		\end{align*}
		
		\end{lemma}

		\begin{proof}
			We let $ t=t_n $ and $ t=t_{n-1} $ in \eqref{varitional-formula} to obtain
\begin{align}\label{exact-solution-eq}
	\begin{split}
			&(\bar{\bm{u}}_t(t_n),\bm{v}_h)+\mathcal A(\bar{\bm{u}}^n,\bm{v}_h)+
			\frac{1}{2}\mathcal B(\int_{0}^{t_n}K(t_n-s)\bm{u}(s) {\rm{d}}s+\int_{0}^{t_{n-1}}K(t_{n-1}-s)\bm{u}(s){\rm{d}}s,\bm{v}_h)\\
			&+\frac{1}{2}(\mathcal C(\bm{u}^n,\bm{u}^n,\bm{v}_h)+\mathcal C(\bm{u}^{n-1},\bm{u}^{n-1},\bm{v}_h))-\mathcal D(\bm{v}_h,\bar{p}^n)=(\bar{\bm{f}}^n,\bm{v}_h),
	\end{split}
\end{align}
where 
\begin{align*}
	\bar{\bm{f}}^n=\frac{\bm{f}(t_n)+\bm{f}(t_{n-1})}{2}.
\end{align*}
Next, introducing the following notations:	
\begin{align*}
	\bm{e}_n=\bm{u}_h^n-\bm{u}_n,\ \bm{\theta}_n=\bm{u}_h^n-\bm{P}_h\bm{u}^n,\ \bm{\xi}_n=\bm{P}_h\bm{u}^n-\bm{u}^n,	
\end{align*}
 and subtracting \eqref{exact-solution-eq} from \eqref{fully-discretization--singular} to obtain 
\begin{align*}
	&(d_t\bm{e}_n,\bm{v}_h)+\mathcal A(\bar{\bm{e}}_n,\bm{v}_h)+\mathcal B(\Delta t^{\alpha}\sum_{p=0}^{n}\omega_p^{(\alpha,\lambda)}\bar{\bm{e}}_{n-p},\bm{v}_h)\\
	&=(\bar{\bm{u}}_t(t_n)-d_t\bm{u}^n,\bm{v}_h)+
	\frac{1}{2}(\mathcal C(\bm{u}^n,\bm{u}^n,\bm{v}_h)+\mathcal C(\bm{u}^{n-1},\bm{u}^{n-1},\bm v_h))-\mathcal C(\bar{\bm{u}}_h^n,\bar{\bm{u}}_h^n,\bm{v}_h)\\
	&+\frac{1}{2}\mathcal B(\int_{0}^{t_n}K(t_n-s)\bm{u}(s) {\rm{d}}s-Q_n^{(\alpha,\lambda)}(\bm{u})+\int_{0}^{t_{n-1}}K(t_{n-1}-s)\bm{u}(s) {\rm{d}}s-Q_{n-1}^{(\alpha,\lambda)}(\bm{u}),\bm{v}_h)\\
	&+\mathcal D(\bm{v}_h,\bar{p}_h^n-\bar{p}^n)+(\bm{f}(\bar{t}_n)-\bar{\bm{f}}^n,\bm{v}_h).
\end{align*}
Noticing that $ \bm{e}_n=\bm{\theta}_n+\bm{\xi}_n $, it follows that
\begin{align}\label{error-eq-fem-singular}
	\begin{split}
			&(d_t\bm{\theta}_n,\bm{v}_h)+\mathcal A(\bar{\bm{\theta}}_n,\bm{v}_h)+\mathcal B(\Delta t^{\alpha}\sum_{p=0}^{n}\omega_{p}^{(\alpha,\lambda)}\bar{\bm{\theta}}_{n-p},\bm{v}_h)=
			(\bar{\bm{u}}_t(t_n)-d_t\bm{u}^n,\bm{v}_h)\\
			&+\frac{1}{2}\mathcal B(\int_{0}^{t_n}K(t_n-s)\bm{u}(s) {\rm{d}}s-Q_n^{(\alpha,\lambda)}(\bm{u})+\int_{0}^{t_{n-1}}K(t_{n-1}-s)\bm{u}(s) {\rm{d}}s-Q_{n-1}^{(\alpha,\lambda)}(\bm{u}),\bm{v}_h)\\
			&+\frac{1}{2}(\mathcal C(\bm{u}^n,\bm{u}^n,\bm{v}_h)+\mathcal C(\bm{u}^{n-1},\bm{u}^{n-1},\bm v_h))-\mathcal C(\bar{\bm{u}}_h^n,\bar{\bm{u}}_h^n,\bm{v}_h)
			+\mathcal D(\bm{v}_h,\bar{p}_h^n-\bar{p}^n)\\
			&+(\bm{f}(\bar{t}_n)-\bar{\bm{f}}^n,\bm{v}_h)
			-\mathcal A(\bar{\bm{\xi}}_n,\bm{v}_h)-\mathcal B(\Delta t^{\alpha}\sum_{p=0}^{n}\omega_{p}^{(\alpha,\lambda)}\bar{\bm{\xi}}_{n-p},\bm v_h).
		\end{split}
\end{align}
	We take $ \bm{v}_h=2\Delta t\bar{\bm{\theta}}_n $ in \eqref{error-eq-fem-singular}  and sum on $ n $ from $ 1 $ to $ N $    to arrive at
	\begin{align*}
		&\|\bm{\theta}_N\|^2-\|\bm{\theta}_0\|^2+2\Delta t\sum_{n=1}^{N}\|\bar{\bm{\theta}}_n\|_a^2
		+2\Delta t^{1+\alpha}\sum_{n=1}^{N}\sum_{p=0}^{n}\omega_{p}^{(\alpha,\lambda)}\mathcal B(\bar{\bm{\theta}}_{n-p},\bar{\bm{\theta}}_n)\\
		&=2\Delta t\sum_{n=1}^{N}(\bar{\bm{u}}_t(t_n)-d_t\bm{u}^n,\bar{\bm{\theta}}_n)
		+\Delta t\sum_{n=1}^{N}\mathcal B(\int_0^{t_n}K(t_n-s)\bm{u}(s) {\rm{d}}s-Q_n^{(\alpha,\lambda)}(\bm{u}),\bar{\bm{\theta}}_n)\\
		&+\Delta t\sum_{n=1}^{N}\mathcal B(\int_0^{t_{n-1}}K(t_{n-1}-s)\bm{u}(s) {\rm{d}}s-Q_{n-1}^{(\alpha,\lambda)}(\bm{u}),\bar{\bm{\theta}}_n)
		+\Delta t\sum_{n=1}^{N}\mathcal C(\bm{u}^n,\bm{u}^n,\bar{\bm{\theta}}_n)\\
		&+\Delta t\sum_{n=1}^{N}\mathcal C(\bm{u}^{n-1},\bm{u}^{n-1},\bar{\bm{\theta}}_n)
		-2\Delta t\sum_{n=1}^{N}\mathcal C(\bar{\bm{u}}_h^n,\bar{\bm{u}}_h^n,\bar{\bm{\theta}}_n)
		+2\Delta t\sum_{n=1}^{N}\mathcal D(\bar{\bm{\theta}}_n,\bar{p}_h^n-\bar{p}^n)\\
		&+2\Delta t\sum_{n=1}^{N}(\bm{f}(\bar{t}_n)-\bar{\bm{f}}^n,\bar{\bm{\theta}}_n)
		-2\Delta t\sum_{n=1}^{N}\mathcal A(\bar{\bm{\xi}}_n,\bar{\bm{\theta}}_n)-2\Delta t^{1+\alpha}\sum_{n=1}^{N}\sum_{p=0}^{n}\omega_p^{(\alpha,\lambda)}\mathcal B(\bar{\bm{\xi}}_{n-p},\bar{\bm{\theta}}_n).
	\end{align*}
			Let $ E  $ be an integer such that $ \displaystyle\|\bar{\bm{\theta}}_E\|=\max_{1 \leq n \leq N}\|\bar{\bm{\theta}}_n\| $, then we have 
			\begin{align*}
				&\|\bm{\theta}_E\|^2-\|\bm{\theta}_0\|^2+2\Delta t\sum_{n=1}^{E}\|\bar{\bm{\theta}}_n\|_a^2
				+2\Delta t^{1+\alpha}\sum_{n=1}^{E}\sum_{p=0}^{n}\omega_{p}^{(\alpha,\lambda)}\mathcal B(\bar{\bm{\theta}}_{n-p},\bar{\bm{\theta}}_n)\\
				&=2\Delta t\sum_{n=1}^{E}(\bar{\bm{u}}_t(t_n)-d_t\bm{u}^n,\bar{\bm{\theta}}_n)+\Delta t\sum_{n=1}^{E}\mathcal B(\int_0^{t_n}K(t_n-s)\bm{u}(s) {\rm{d}}s-Q_n^{(\alpha,\lambda)}(\bm{u}),\bar{\bm{\theta}}_n)\\
				&+\Delta t\sum_{n=1}^{E}\mathcal B(\int_0^{t_{n-1}}K(t_{n-1}-s)\bm{u}(s) {\rm{d}}s-Q_{n-1}^{(\alpha,\lambda)}(\bm{u}),\bar{\bm{\theta}}_n)+\Delta t\sum_{n=1}^{E}\mathcal C(\bm{u}^n,\bm{u}^n,\bar{\bm{\theta}}_n)\\
				&+\Delta t\sum_{n=1}^{E}\mathcal C(\bm{u}^{n-1},\bm{u}^{n-1},\bar{\bm{\theta}}_n)
				-2\Delta t\sum_{n=1}^{E}\mathcal C(\bar{\bm{u}}_h^n,\bar{\bm{u}}_h^n,\bar{\bm{\theta}}_n)
				+2\Delta t\sum_{n=1}^{E}\mathcal D(\bar{\bm{\theta}}_n,\bar{p}_h^n-\bar{p}^n)\\
				&+2\Delta t\sum_{n=1}^{E}(\bm{f}(\bar{t}_n)-\bar{\bm{f}}^n,\bar{\bm{\theta}}_n)
				-2\Delta t\sum_{n=1}^{E}\mathcal A(\bar{\bm{\xi}}_n,\bar{\bm{\theta}}_n)-2\Delta t^{1+\alpha}\sum_{n=1}^{E}\sum_{p=0}^{n}\omega_p^{(\alpha,\lambda)}\mathcal B(\bar{\bm{\xi}}_{n-p},\bar{\bm{\theta}}_n)\\
				&=\sum_{i=1}^{10}R_i.
			\end{align*}
			Next, we turn to estimate $ \{R_i\}_{i=1}^{10} $. For the term $ R_1 $, we use integration by parts and simple computation to obtain 
			\begin{align*}
				\Delta t\bar{\bm{u}}_t(t_1)-\Delta td_t\bm{u}^1&=\frac{\Delta t}{2}(\bm{u}_t(t_1)+\bm{u}_t(t_0))-\int_{0}^{\Delta t}\bm{u}_t(t) {\rm{d}}t\\
				&=\int_{0}^{\Delta t}(\int_t^{\Delta t}\frac{1}{2}\bm{u}_{tt}(s) {\rm{d}}s+\int_{t}^{0}\frac{1}{2}\bm{u}_{tt}(s) {\rm{d}}s) {\rm{d}}t\\
				&\le \Delta t\int_0^{\Delta t}|\bm{u}_{tt}(s)| {\rm{d}}s,\\
				\Delta t\bar{\bm{u}}_t(t_n)-\Delta td_t\bm{u}^n&=
				\Delta t\bar{\bm{u}}_t(t_n)-\int_{t_{n-1}}^{t_n}\bm{u}_t(s) {\rm{d}}s\\
				&=\frac{1}{2}\int_{t_{n-1}}^{t_n}(t-t_{n-1})(t_n-t)\bm{u}_{ttt}(s) {\rm{d}}s\\
				&\le C\Delta t^2\int_{t_{n-1}}^{t_n}|\bm{u}_{ttt}(s)| {\rm{d}}s,
			\end{align*}
			then we can estimate $ R_1 $ by Young inequality as follows:
			\begin{align*}
				R_1&=2\Delta t(\bar{\bm{u}}_t(t_1)-d_t\bm{u}^1,\bar{\bm{\theta}}_1)
				+2\Delta t\sum_{n=2}^{E}(\bar{\bm{u}}_t(t_n)-d_t\bm{u}^n,\bar{\bm{\theta}}_n)\\
				&\le C(\Delta t\int_{0}^{\Delta t}\|\bm{u}_{tt}(s)\| {\rm{d}}s+\Delta t^2\int_{\Delta t}^{t_E}\|\bm{u}_{ttt}(s)\| {\rm{d}}s)\|\bar{\bm{\theta}}_E\|\\
				&\le C\Delta t^2(\int_{0}^{\Delta t}\|\bm{u}_{tt}(s)\| {\rm{d}}s)^2+\Delta t^4(\int_{\Delta t}^{t_E}\|\bm{u}_{ttt}(s)\| {\rm{d}}s)^2+C_1\|\bar{\bm{\theta}}_E\|^2.
			\end{align*}
			It follows from \cref{quadrature-error}   that
			\begin{align*}
				R_2+R_3&\le C\Delta t\sum_{n=1}^{E}(\Delta t^2t_n^{\alpha-1}\|\bm{u}_t(0)\|_a
				+\Delta t^{\alpha+1}\int_{t_{n-1}}^{t_n}\|\bm{u}_{tt}(s)\|_a {\rm{d}}s\\
				&+\Delta t^2\int_{0}^{t_{n-1}}(t_n-s)^{\alpha-1}\|\bm{u}_{tt}(s)\|_a\ {\rm{d}}s)\|\bar{\bm{\theta}}_n\|_a\\
				&\le C\Delta t^5\sum_{n=1}^{E}\|\bm{u}_t(0)\|_a^2
				+C\Delta t^{2\alpha+3}\sum_{n=1}^{E}(\int_{t_{n-1}}^{t_n}\|\bm{u}_{tt}\|_a{\rm d}s)^2\\
				&+C\Delta t^5\sum_{n=1}^{E}(\int_{0}^{t_{n-1}}(t_n-s)^{\alpha-1}\|\bm{u}_{tt}(s)\|_a\ {\rm{d}}s)^2
				+\frac{\mu\Delta t}{6}\sum_{n=1}^{E}\|\bar{\bm{\theta}}_n\|_a^2.
			\end{align*}
			Using the same technique in the proof in \cref{error estimate-fem--1} for the terms $ R_4+R_5+R_6 $, we can reach 
			\begin{align*}
				\sum_{i=4}^6R_i&\le C\Delta t^5\sum_{n=1}^{E}\|\nabla(d_t\bm{u}^n)\|^4+\frac{\mu\Delta t}{6}\sum_{n=1}^{E}\|\bar{\bm{\theta}}_n\|_a^2+C\Delta t\sum_{n=1}^{E}h^{2k}|\bm{u}^n|_{k+1}^2\\
				&\quad+C\Delta t\sum_{n=1}^{E}\|\bar{\bm{\theta}}_n\|^2
				+C\Delta t^4\sum_{n=1}^{E}\int_{t_{n-1}}^{t_n}(\|\bm{u}_{tt}\|_0^2\|\bm{u}\|_0^2+
					\|\bm{u}_t\|_2^2\|\bm{u}_t\|_0^2){\rm d}t.
			\end{align*}
		Using the fact
	\begin{align*}
		\mathcal{D}(\bar{\bm{\theta}}_n,\bar{p}_h^n)=0=\mathcal{D}(\bar{\bm{\theta}}_n,\rho_h(\bar{p}^n)),
	\end{align*}	
	
			we use the property of projection of $ \rho_h $  to obtain
			\begin{align*}
				R_7&=2\Delta t\sum_{n=1}^{E}\mathcal D(\bar{\bm{\theta}}_n,\rho_h(\bar{p}^n)-\bar{p}^n)
				\le C\Delta t\sum_{n=1}^{E}h^{l+1}\|\bar{\bm{\theta}}_n\|_a|p^n|_{l+1}\\
				&\le C\Delta t\sum_{n=1}^{E}h^{2l+2}|p^n|_{l+1}^2+\frac{\mu\Delta t}{6}\sum_{n=1}^{E}\|\bar{\bm{\theta}}_n\|_a^2.
			\end{align*}
			 Young inequality applying to $R_8$ and $R_9$ to deduce that 
			\begin{align*}
				R_8&=2\Delta t\sum_{n=1}^{E}(\bm{f}(\bar{t}_n)-\bar{\bm{f}}^n,\bar{\bm{\theta}}_n))\\
				&=-\Delta t\sum_{n=1}^{E}(\int_{t_{n-1}}^{\bar{t}_n}(t-t_{n-1})\bm{f}_{tt}\ {\rm{d}}t+\int_{\bar{t}_n}^{t_n}(t_n-t)\bm{f}_{tt}\ {\rm{d}}t,\bar{\bm{\theta}}_n)\\
				&\le C\Delta t^{5/2}\sum_{n=1}^{E}(\int_{t_{n-1}}^{t_n}\|\bm{f}_{tt}\|^2{\rm{d}}t)^{1/2}\|\bar{\bm{\theta}}_n\|\\
				&\le C\Delta t^4\int_0^{t_E}\|\bm{f}_{tt}\|^2{\rm{d}}t+\frac{\mu\Delta t}{6}\sum_{n=1}^{E}\|\bar{\bm{\theta}}_n\|_a^2,\\
				R_9&\le C\Delta t\sum_{n=1}^{E}\|\bar{\bm{\theta}}_n\|_ah^k|\bm{u}^n|_{k+1}
				\le C\Delta t\sum_{n=1}^{E}h^{2k}|\bm{u}^n|_{k+1}^2+\frac{\mu\Delta t}{6}\sum_{n=1}^{E}\|\bar{\bm{\theta}}_n\|_a^2.
			\end{align*}
			For the term $ R_{10} $, we know from \cite[Lemma 3]{MR4626456} that $ \omega_p^{(\alpha,\lambda)}=\mathcal{O}(\Delta t^{1-\alpha}) $, so we can obtain
			\begin{align*}
				R_{10}&\le C\Delta t\sum_{n=1}^{E}h^k|\bm{u}^n|_{k+1}\|\bar{\bm{\theta}}_n\|_a\le  C\Delta t\sum_{n=1}^{E}h^{2k}|\bm{u}^n|_{k+1}^2+\frac{\mu\Delta t}{6}\sum_{n=1}^{E}\|\bar{\bm{\theta}}_n\|_a^2.
			\end{align*}
			Combine all estimations for $ \{R_i\}_{i=1}^{10} $ and notice that $ \|\theta_0\|\le h^k|\bm{u}(0)|_{k+1} $, it follows that 
			\begin{align*}
				&\|\bm{\theta}_E\|^2-\|\bm{\theta}_0\|^2+2\Delta t\sum_{n=1}^{E}\|\bar{\bm{\theta}}_n\|_a^2+2\Delta t^{1+\alpha}\sum_{n=1}^{E}\sum_{p=0}^{n}\omega_p^{(\alpha,\lambda)}\mathcal B(\bar{\bm{\theta}}_{n-p},\bar{\bm{\theta}}_n)\\
				&\le C\Delta t^2(\int_0^{\Delta t}\|\bm{u}_{tt}(s)\| \ {\rm d}s)^2
				+C\Delta t^4(\int_{\Delta t}^{t_E}\|\bm{u}_{ttt}(s)\| \ {\rm d}s)^2+C\Delta t^5\sum_{n=1}^{E}\|\bm{u}_t(0)\|_a^2\\
				&\quad+ C\Delta t^{2\alpha+3}\sum_{n=1}^{E}(\int_{t_{n-1}}^{t_n}\|\bm{u}_{tt}(s)\|_a \ {\rm d}s)^2
				+C\Delta t^5\sum_{n=1}^{E}(\int_{0}^{t_{n-1}}(t_n-s)^{\alpha-1}\|\bm{u}_{tt}\|_a \ {\rm d}s)^2\\
				&\quad +C_1\Delta t\|\bm{\theta}_E\|^2
				+\mu\Delta t\sum_{n=1}^{E}\|\bar{\bm{\theta}}_n\|_a^2+C\Delta t^5\sum_{n=1}^{E}\|\nabla(d_t\bm{u}^n)\|^4
				+C\Delta t\sum_{n=1}^{E}h^{2k}|\bm{u}^n|_{k+1}^2
				\\
				&\quad+C\Delta t\sum_{n=1}^{E}h^{2l+2}|p^n|_{l+1}^2
					+C\Delta t^4\sum_{n=1}^{E}\int_{t_{n-1}}^{t_n}(\|\bm{u}_{tt}\|_0^2\|\bm{u}\|_0^2+
					\|\bm{u}_t\|_2^2\|\bm{u}_t\|_0^2){\rm d}t
				.
			\end{align*}
		By applying \cref{positive-define property} and the same technique for the right hand term $ C\Delta t^5\sum_{n=1}^{E}\|\nabla(d_t\bm{u}^n)\|^4 $, we can  get
		\begin{align*}
			&C\Delta t^5\sum_{n=1}^{E}\|\nabla(d_t\bm{u}^n)\|^4 \le C\Delta t^4\int_0^T\|\partial_t(\nabla\bm{u})\|^4dt,\\
			&2\Delta t^{1+\alpha}\sum_{n=1}^{E}\sum_{p=0}^{n}\omega_p^{(\alpha,\lambda)}\mathcal B(\bar{\bm{\theta}}_{n-p},\bar{\bm{\theta}}_n)\ge 0.
		\end{align*}
	By using regularity condition for $ \bm{u} $ in \eqref{regularity-condition},  we can obtain 
	\begin{align}\label{eq6--10}
			&\|\bm{\theta}_E\|^2-\|\bm{\theta}_0\|^2+2\Delta t\sum_{n=1}^{E}\|\bar{\bm{\theta}}_n\|_a^2+2\Delta t^{1+\alpha}\sum_{n=1}^{E}\sum_{p=0}^{n}\omega_p^{(\alpha,\lambda)}\mathcal B(\bar{\bm{\theta}}_{n-p},\bar{\bm{\theta}}_n)\notag \\
			&\le C\Delta t^{2+2\alpha}+C_1\Delta t\|\bm{\theta}_E\|^2+\mu\Delta t\sum_{n=1}^{E}\|\bar{\bm{\theta}}_n\|_a^2+C\Delta t\sum_{n=1}^{E}h^{2k}|\bm{u}^n|_{k+1}^2+C\Delta t\sum_{n=1}^{E}h^{2l+2}|p^n|_{l+1}^2.
	\end{align}
Finally	we choose $ C_1\in(0,\frac{1}{2}) $ and $ \mu\in(0,1) $, we can reach our final conclusion by  applying Gronwall's inequality to \eqref{eq6--10} . 
		\end{proof}
		
		Finally, we can derive  error bound between the solutions of \eqref{hat-euqation-singular} and \eqref{OldryodModule-general-singular}.
		
		\begin{theorem}\label{main_result2}
		
		Let $(\bm{u}(t),p(t))$ and $(\widehat{\bm{u}}_h^n,\widehat p_h^n)$ be the solutions of \eqref{OldryodModule-general-singular} and \eqref{hat-euqation-singular}, respectively. Assume that the tolerance $\mathtt{tol}$ is used in both $p$-truncation and singular-value truncation throughout the incremental SVD procedure. Under \Cref{assumption-singular}, if
		\[
		\bm u\in L^\infty(0,T;H^{k+1}(\Omega)^d),\qquad p\in L^\infty(0,T;H^{l+1}(\Omega)),
		\]
		and $\bm u$ satisfies \eqref{regularity-condition}, then there exists a constant $C_T>0$, independent of $h$, $\Delta t$, and $\mathtt{tol}$, such that
		\begin{align*}
			\|\bm{u}(t_n)-\widehat{\bm{u}}_h^n\|
			\le C_T\Bigl(h^{\min\{k+1,l+1\}}+\Delta t^{1+\alpha}+(1+T_{\rm sv})\sqrt{\sigma(S)}\,\mathtt{tol}\Bigr).
		\end{align*}
		
		\end{theorem}

		\begin{proof}
			We notice that $ (\bm{u}_h^{n},p_h^{n}) $ and $ (\widehat{\bm{u}}_h^{n},\widehat p_h^{n}) $ satisfying  equations \eqref{fully-discretization--singular} and \eqref{hat-euqation-singular}, introducing  the following notations:
			\begin{align*}
				\widehat{\bm{e}}_i=\bm{u}_h^i-\widehat{\bm{u}}_h^i ,\  \widetilde{\bm{e}}_{i,j}=\bm{u}_h^j-\widetilde{\bm{u}}_h^{i,j},
			\end{align*}  
			subtracting  \eqref{hat-euqation-singular} from \eqref{fully-discretization--singular} to obtain
			\begin{align}\label{error_hat_tilde1-singular}
				\begin{split}
					&(d_t\widehat{\bm e}_{n},\bm{v}_h)
					+\mathcal A(\bar{\widehat{\bm{e}}}_n,\bm{v}_h)
					+\mathcal B
					(\Delta t^{\alpha}\sum_{p=1}^{n}\omega_p^{(\alpha,\lambda)}\bar{\widetilde{\bm e}}_{n-1,n-p}
					+\Delta t^{\alpha}\omega_0^{(\alpha,\lambda)}\bar{\widehat{\bm{e}}}_n+\bar{\rho}_n^{(\alpha,\lambda)}\widehat{\bm{e}}_0,\bm{v}_h)\\
					&\quad-\mathcal D(\bm{v}_h,\bar{p}_h^n-\bar{\widehat{p}}_h^n)+\mathcal C(\bar{\bm{u}}_h^n,\bar{\bm{u}}_h^n,\bm{v}_h)-\mathcal C(\bar{\widehat{\bm{u}}}_h^n,\bar{\widehat{\bm{u}}}_h^n,\bm{v}_h)=0.
				\end{split}
			\end{align}
			Taking $ \bm{v}_h=2\Delta t\widehat{\bm{e}}_{n} $ in $\eqref{error_hat_tilde1-singular}$, summing on $ n $ from $ n=1 $ to $ N $  and noticing $ \widehat{\bm{e}}_0=\bm{0} $, it follows that
			\begin{align}
				\begin{split}\label{sum-singular-isvd-error}
					&\|\widehat{\bm{e}}_N\|^2+2\Delta t\sum_{n=1}^{N}\|\widehat{\bm{e}}_n\|_a^2+2\Delta t^{1+\alpha}\sum_{n=1}^{N}\mathcal B(\sum_{p=0}^{n}\omega_p^{(\alpha,\lambda)}\bar{\widehat{\bm{e}}}_{n-p},\bar{\widehat{\bm{e}}}_n)\\
					&=2\Delta t^{1+\alpha}\sum_{n=1}^{N}\mathcal B(\sum_{p=1}^{n}\omega_p^{(\alpha,\lambda)}(\bar{\widehat{\bm{e}}}_{n-p}-\bar{\widetilde{\bm{e}}}_{n-1,n-p}),\bar{\widehat{\bm e}}_{n})\\
					&\quad +2\Delta t\sum_{n=1}^{N}\mathcal C(\bar{\widehat{\bm{u}}}_h^n,\bar{\widehat{\bm{u}}}_h^n,\bar{\widehat{\bm{e}}}_n)-2\Delta t\sum_{n=1}^{N}\mathcal C(\bar{\bm{u}}_h^n,\bar{\bm{u}}_h^n,\bar{\widehat{\bm{e}}}_n)\\
					&=R_1+R_2+R_3.
				\end{split}
			\end{align}
			Now we move to bound the terms $ \{R_i\}_{i=1}^{3} $. From \cite[Lemma 3]{MR4626456}, We know that $\omega_n^{(\alpha,\lambda)}=\mathcal{O}(n^{\alpha-1})  $.  Then by using Young inequality, we can estimate the term $ R_1 $ as follows:
			\begin{align*}
				R_1&=-2\Delta t^{1+\alpha}\sum_{n=1}^{N}\sum_{p=1}^{n}\omega_{p}^{(\alpha,\lambda)}\mathcal B(\bar{\widehat{\bm{e}}}_{n-p}-\bar{\widetilde{\bm{e}}}_{n-1,n-p},\bar{\widehat{\bm{e}}}_n)\\
				&\le C\Delta t\sum_{n=1}^{N}\max_{1 \leq j \leq n}\|\bar{\widehat{\bm{e}}}_{n-j}-\bar{\widetilde{\bm{e}}}_{n-1,n-j}\|_a\|\bar{\widehat{\bm{e}}}_n\|_a\\
				&\le C\Delta t\sum_{n=1}^{N}\max_{1 \leq j \leq n}\|\bar{\widehat{\bm{e}}}_{n-j}-\bar{\widetilde{\bm{e}}}_{n-1,n-j}\|_a^2
				+\frac{\mu}{2}\Delta t\sum_{n=1}^{N}\|\widehat{\bm{e}}_n\|_a^2\\
				&\le CT\max_{1 \leq n \leq N}\max_{1 \leq j \leq n}\|\bar{\widehat{\bm{e}}}_{n-j}-\bar{\widetilde{\bm{e}}}_{n-1,n-j}\|_a^2
				+\frac{\mu}{2}\Delta t\sum_{n=1}^{N}\|\widehat{\bm{e}}_n\|_a^2,
			\end{align*}
			where $ \mu\in (0,1) $ is a positive constant and will be specified later. 
			For the term $ R_2+R_3 $, we use the same technique used in \cref{main_result} to estimate to obtain 
			\begin{align*}
				R_2+R_3=-2\Delta t\sum_{n=1}^{N}\mathcal C(\bar{\widehat{\bm{e}}}_n,\bar{\bm{u}}_h^n,\bar{\widehat{\bm{e}}}_n)
				\le C\Delta t\sum_{n=1}^{N}\|\bar{\widehat{\bm{e}}}_n\|^2+\frac{\mu}{2}\Delta t\sum_{n=1}^{N}\|\bar{\widehat{\bm{e}}}_n\|_a^2.
			\end{align*}
			Since the sequences $ \omega_{p}^{(\alpha,\lambda)} $ satisfy  the positive definite condition \eqref{oprator-sequence-positive}, we choose $ \mu =\frac{1}{2} $ and apply Gronwall's inequality to \eqref{sum-singular-isvd-error} to derive 
			\begin{align}\label{last-inequality}
				\|\widehat{\bm{e}}_N\|^2\le CTe^T\max_{1\le n\le N}\max_{1 \leq j \leq n}\|\bar{\widehat{\bm{e}}}_{n-j}-\bar{\widetilde{\bm{e}}}_{n-1,n-j}\|_a^2.	
			\end{align}
			Finally, applying \Cref{hat-tilde-error} to \eqref{last-inequality} to give our result.
		\end{proof}

		\section{Numerical experiments}\label{section 7}
		
		This section reports three numerical experiments illustrating the behavior of the compressed schemes \eqref{hat_equation-nonsingular} and \eqref{hat-euqation-singular}. The first two are manufactured-solution tests for nonsingular and weakly singular kernels, respectively, and are used to verify accuracy. The third is a planar $4{:}1$ contraction benchmark, included to check that the compression does not materially alter a representative viscoelastic flow field. Throughout, we use the Mini element ($P_{1b}$--$P_1$) for space discretization and Crank--Nicolson time stepping.

\begin{example}\label{numerical example--1}
			In this example, we  evaluate the performance of our new method
			concerning the nonsingular kernel $K(t)=\ln{(1+t)}  $ for the following equation:
			\begin{align*}
				\bm{u}_t-10\Delta \bm{u}-25\int_{0}^{t}\ln{(1+t-s)}\Delta \bm{u}(s) {\rm d}s+(\bm{u}\cdot\nabla)\bm{u}+\nabla p=\bm{f},
			\end{align*}
			with exact solution
			\begin{align*}
				u_1(x,y,t)&=5tx^2(x-1)^2y(y-1)(2y-1)+4\sin^2{(\pi x)}\sin{(\pi y)}\cos{(\pi y)},\\
				u_2(x,y,t)&=-5tx(x-1)(2x-1)y^2(y-1)^2-4\sin{(\pi x)}\cos{(\pi x)}\sin^2{(\pi y)},\\
				p(x,y,t)&=10(2x-1)(2y-1)\cos t.
			\end{align*}
		\begin{table}[h]
			\centering
			\begin{tabular}{c|c|c|c|c|c}
				\Xhline{1pt}
				$\frac{\sqrt{2}}{h}$  & $ \|\bm u_h^N-\bm u(T)\| $ & rate & $ \|\widehat {\bm u}_h^N-\bm u(T)\| $ & rate & $ \|\bm u_h^N-\widehat{\bm u}_h^N\| $\\ \Xhline{1pt}
				$ 20 $  &2.0946E-02     & -   & 2.0946E-02     & -   & 1.7721E-12      \\ \hline
				$ 30 $  & 9.3613E-03    & 1.9683   & 9.3613E-03      &1.9683  & 1.0597E-12      \\ \hline
				$ 40 $ &5.2699E-03      & 1.9972   & 5.2699E-03      & 1.9972   & 1.3167E-12     \\ \hline
				$ 50 $  & 3.3730E-03     & 1.9996    & 3.3730E-03      & 1.9996   & 2.3583E-12     \\ \hline
				$ 60$  & 2.3422E-03     & 2.0004    & 2.3422E-03     & 2.0004    & 1.0094E-12       \\ \hline
				$ 70 $  & 1.7205E-03       & 2.0011   &  1.7205E-03     & 2.0011    & 5.4567E-12      \\ \hline
				$ 80 $  &1.3171E-03       & 2.0009    & 1.3171E-03      & 2.0009   & 3.0597E-14     \\ \hline
				$ 90 $ &1.0405E-03      & 2.0014    &1.0405E-03     & 2.0014   & 5.7959E-12      \\ \hline
				$ 100 $ & 8.4277E-04      & 2.0004    &  8.4277E-04      & 2.0004  & 1.2823E-11      \\ \hline
				$ 110 $ &6.9643E-04      &2.0011    &6.9643E-04     & 2.0011   & 2.7895E-12     \\  \Xhline{1pt}
			\end{tabular}
			\caption{The convergence rates of velocity $ \bm{u} $ at final time $ T=1 $:   $\| \bm u_h^N-\bm u(T)\|$, $\|\widehat{\bm  u}_h^N-\bm u(T) \|$  and $ \|\bm{ u}_h^N-\widehat{\bm u}_h^N\| $ for $\Delta t = \frac{1}{2}h$  with  kernel $ K(t)=25\ln {(1+t)} $.}
			\label{new-table1}
		\end{table}
			\begin{table}[h]
			\centering
			\begin{tabular}{c|c|c|c|c|c}
				\Xhline{1pt}
				$\frac{\sqrt{2}}{h}$  & $ \|p_h^N-p(T)\| $ & rate & $ \|\widehat {p}_h^N-p(T)\| $ & rate & $ \|p_h^N-\widehat{p}_h^N\| $\\ \Xhline{1pt}
				$ 20 $  & 3.2636     & -   & 3.2636     & -   & 8.3330E-13       \\ \hline
				$ 30 $  & 1.7818    & 1.4926   &1.7818      &1.4926  & 1.9434E-12      \\ \hline
				$ 40 $ & 1.1666       &1.4722   & 1.1666      & 1.4722   & 2.5994E-12     \\ \hline
				$ 50 $  & 0.8433     &1.4543    & 0.8433       &1.4543   & 3.8030E-12     \\ \hline
				$ 60$  & 0.6487     & 1.4390   & 0.6487      & 1.4390   & 5.2958E-12       \\ \hline
				$ 70 $  &0.5206     & 1.4271   & 0.5206    & 1.4271    & 7.1786E-12      \\ \hline
				$ 80 $  &0.4309       & 1.4162    & 0.4309     & 1.4162  &8.7950E-12     \\ \hline
				$ 90 $ &0.3652      & 1.4045   & 0.3652      & 1.4045   &1.1529E-11      \\ \hline
				$ 100 $ & 0.3153      & 1.3945   & 0.3153    & 1.3945   &1.1225E-11       \\ \hline
				$ 110 $ &0.2763      &1.3853    & 0.2763     & 1.3853   & 1.7898E-11   \\   \Xhline{1pt}
			\end{tabular}
			\caption{The convergence rates of pressure $ p $ at final time $ T=1 $:   $\| p_h^N-p(T)\|$, $\|\widehat{p}_h^N-p(T) \|$  and $ \|p_h^N-\widehat{p}_h^N\| $ for $\Delta t = \frac{1}{2}h$   with  kernel $ K(t)=25\ln {(1+t)} $.}
			\label{new-table2}
		\end{table}
		\begin{figure}[h]
			\centering
			\begin{minipage}{0.49\linewidth}
				\centering
				\includegraphics[width=0.9\linewidth]{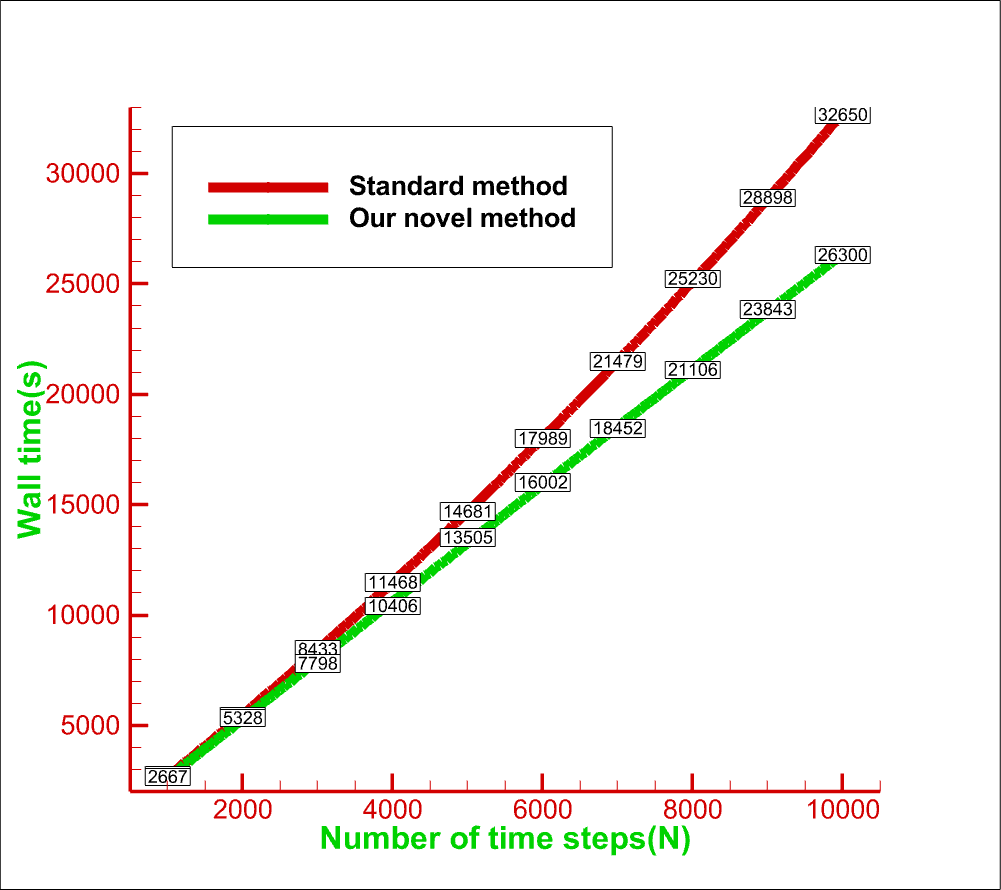}
			\end{minipage}
			\begin{minipage}{0.49\linewidth}
				\centering
				\includegraphics[width=0.9\linewidth]{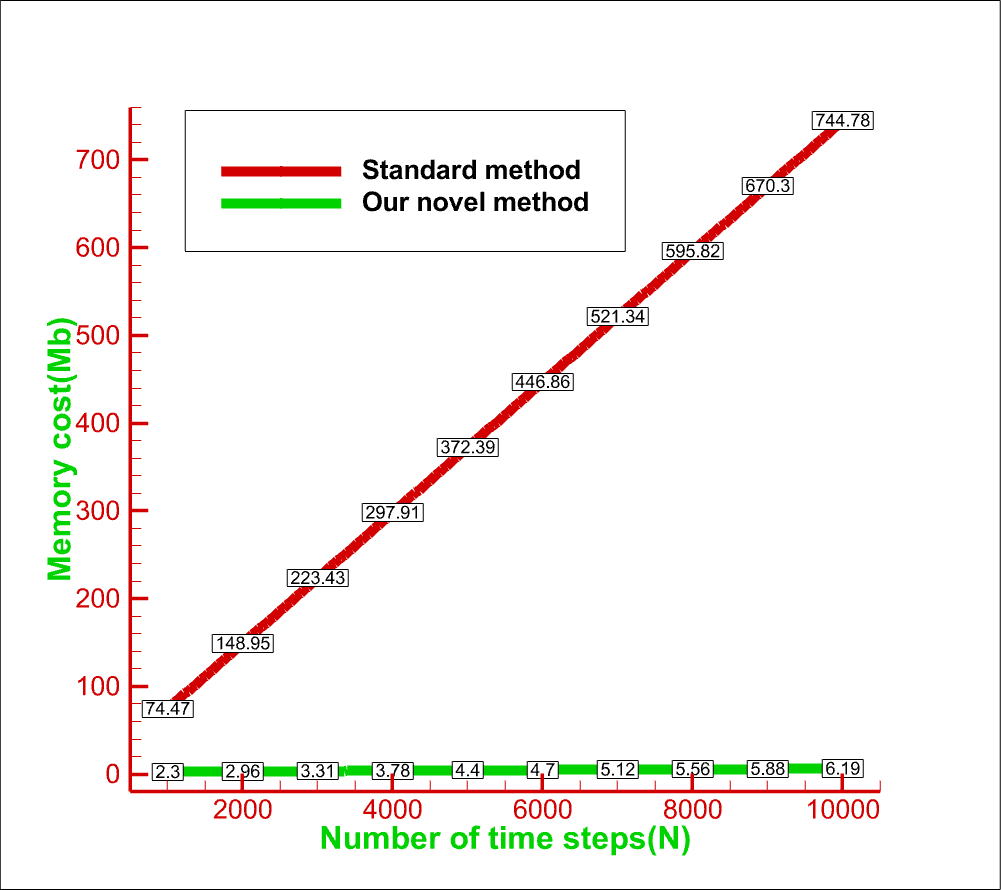}
			\end{minipage}
			\caption{A comparison of wall time and memory costs between the two algorithms is conducted for various time steps and $ K(t)=\ln {(1+t)} $, specifically when $h=1/50$ and $\Delta t=10^{-4}$.}
			\label{figure--1}
		\end{figure}

			For \Cref{numerical example--1}, \Cref{assumption-1} is readily verified. Tables~\ref{new-table1} and \ref{new-table2} report the $L^2$-errors for velocity and pressure at the final time $T=1$ with $\Delta t=\frac12 h$. The standard fully discrete scheme exhibits the expected second-order behavior for the velocity variable in this smooth manufactured example, while the compressed scheme with $\mathtt{tol}=10^{-12}$ is numerically indistinguishable from it. The discrepancy $\|\bm u_h^N-\widehat{\bm u}_h^N\|$ remains at the level predicted by the tolerance-dependent perturbation analysis, and for this tolerance is close to machine precision in practice. Figure~\ref{figure--1} shows the accompanying reductions in wall-clock time and memory footprint.
			
		\end{example}
		\begin{example}
			In this scenario, we contemplate the equation with  a weakly singular kernel.  The equation is given by 
			\begin{align*}
				\bm{u}_t-\Delta \bm{u}-\int_{0}^{t}K(t-s)\Delta \bm{u}(s) {\rm d}s+(\bm{u}\cdot\nabla)\bm{u}+\nabla p=\bm{f},
			\end{align*}
			with exact solution $ \bm{u}(x,y,t)=(u_1(x,y,t),u_2(x,y,t)) ^T$ and $ p(x,y,t) $ defined as follows:
			\begin{align*}
				u_1(x,y,t)&=-10x^2(x-1)^2y(y-1)(2y-1)\frac{t^{2+\alpha}}{\Gamma(3+\alpha)}e^{-\lambda t},\\
				u_2(x,y,t)&=10x(x-1)(2x-1)y^2(y-1)^2\frac{t^{2+\alpha}}{\Gamma(3+\alpha)}e^{-\lambda t},\\
				p(x,y,t)&=10(2x-1)(2y-1)\cos t,
			\end{align*}
			where 
			\begin{align*}
				K(t)=e^{-\lambda t}\frac{1}{\Gamma(\alpha)}t^{\alpha-1}, \ \alpha=0.5,\ \lambda=0.5.
			\end{align*}

			For this weakly singular example we again use $\mathtt{tol}=10^{-12}$ in the compressed scheme and choose $\Delta t=\mathcal O(h)$, specifically $h=\frac14\Delta t$. The generic theory in Section~\ref{section 6} yields an error bound of order $\Delta t^{1+\alpha}$ for $\alpha\in(0,1)$. The manufactured solution used here is smoother than the worst-case setting covered by the theorem, and the observed velocity rate in Table~\ref{new-table3} is empirically close to second order. As in the nonsingular test, the compressed and uncompressed schemes are nearly indistinguishable at the reported tolerance; see Tables~\ref{new-table3}--\ref{new-table4}. Figure~\ref{figure--2} shows the accompanying savings in wall-clock time and memory.
			
			\begin{table}[h]
				\centering
				\begin{tabular}{c|c|c|c|c|c}
					\Xhline{1pt}
					$\frac{\sqrt{2}}{h}$  & $ \|\bm u_h^N-\bm u(T)\| $ & rate & $ \|\widehat {\bm u}_h^N-\bm u(T)\| $ & rate & $ \|\bm u_h^N-\widehat{\bm u}_h^N\| $\\ \Xhline{1pt}
					$ 20 $  & 1.2841E-04      & -   & 1.2841E-04     & -   & 3.2230E-14       \\ \hline
					$ 30 $  & 5.6702E-05     & 2.0159   & 5.6702E-05      &2.0159   & 1.9605E-14      \\ \hline
					$ 40 $ & 3.1774E-05       & 2.0131   & 3.1774E-05      & 2.0131   & 8.9601E-15     \\ \hline
					$ 50 $  & 2.0284E-05     & 2.0113    & 2.0284E-05      & 2.0113   & 1.1126E-14      \\ \hline
					$ 60$  & 1.4062E-05     & 2.0094    & 1.4062E-05     & 2.0094    & 9.3910E-15       \\ \hline
					$ 70 $  & 1.0318E-05       & 2.0083    & 1.0318E-05    & 2.0083    & 1.1573E-14      \\ \hline
					$ 80 $  & 7.8923E-06       & 2.0070    & 7.8923E-06      & 2.0070   & 1.0211E-14     \\ \hline
					$ 90 $ &6.2311E-06       & 2.0065    & 6.2311E-06      & 2.0065   & 8.6845E-15       \\ \hline
					$ 100 $ & 5.0440E-06       & 2.0060    & 5.0440E-06      & 2.0060   & 8.6372E-15       \\ \hline
					$ 110 $ & 4.1664E-06       & 2.0055    & 4.1664E-06      & 2.0055   & 7.9131E-15      \\ \hline
					$120 $ & 3.4995E-06      & 2.0047   & 3.4995E-06      & 2.0047   & 6.4789E-15       \\  \Xhline{1pt}
				\end{tabular}
				\caption{The convergence rates of velocity $ \bm{u} $ at final time $ T=1 $:   $\| \bm u_h^N-\bm u(T)\|$, $\|\widehat{\bm  u}_h^N-\bm u(T) \|$  and $\| \bm{u}_h^N-\widehat{\bm u}_h^N\| $ for $\Delta t = \frac{1}{4}h$  with singular kernel.}
				\label{new-table3}
			\end{table}
		
			\begin{table}[h]
				\centering
				\begin{tabular}{c|c|c|c|c|c}
					\Xhline{1pt}
					$\frac{\sqrt{2}}{h}$  & $ \|p_h^N-p(T)\| $ & rate & $ \|\widehat{p}_h^N-p(T)\| $ & rate & $ \|p_h^N-\widehat{p}_h^N\| $\\ \Xhline{1pt}
					$ 20 $  & 9.2336E-03      & -   & 9.2336E-03     & -   & 1.9720E-14       \\ \hline
					$ 30 $  & 4.1973E-03     & 1.9440   & 4.1973E-03      &1.9440  & 1.7349E-14      \\ \hline
					$ 40 $ & 2.4283E-03       & 1.9022   & 2.4283E-03     & 1.9022   & 1.6742E-14    \\ \hline
					$ 50 $  & 1.6005E-03     & 1.8682    & 1.6005E-03      & 1.8682   & 7.8335E-15      \\ \hline
					$ 60$  & 1.1443E-03    &1.8399    & 1.1443E-03     & 1.8399    &2.2512E-14       \\ \hline
					$ 70 $  & 8.6491E-04       & 1.8164    & 8.6491E-04    & 1.8164    & 1.9549E-14      \\ \hline
					$ 80 $  & 6.8047E-04       & 1.7961    & 6.8047E-04      & 1.7961   & 2.1638E-14     \\ \hline
					$ 90 $ &5.5187E-04       & 1.7784    & 5.5187E-04      & 1.7784   & 5.6711E-14       \\ \hline
					$ 100 $ &4.5832E-04      & 1.7629    &4.5832E-04      & 1.7629   & 4.9709E-14       \\ \hline
					$ 110 $ & 3.8795E-04       & 1.7489    & 3.8795E-04      & 1.7489   & 6.3094E-14      \\ \hline
					$120 $ & 3.3355E-04       & 1.7364   & 3.3355E-04     & 1.7364   & 7.6639E-14       \\  \Xhline{1pt}
				\end{tabular}
				\caption{The convergence rates of pressure $p $ at final time $ T=1 $:   $\| p_h^N-p(T)\|$, $\|\widehat{p}_h^N-p(T) \|$  and $ \|p_h^N-\widehat{p}_h^N\| $ for $\Delta t = \frac{1}{4}h$  with singular kernel.}
				\label{new-table4}
			\end{table}
		
		\begin{figure}[h]
			\centering
			\begin{minipage}{0.49\linewidth}
				\centering
				\includegraphics[width=0.9\linewidth]{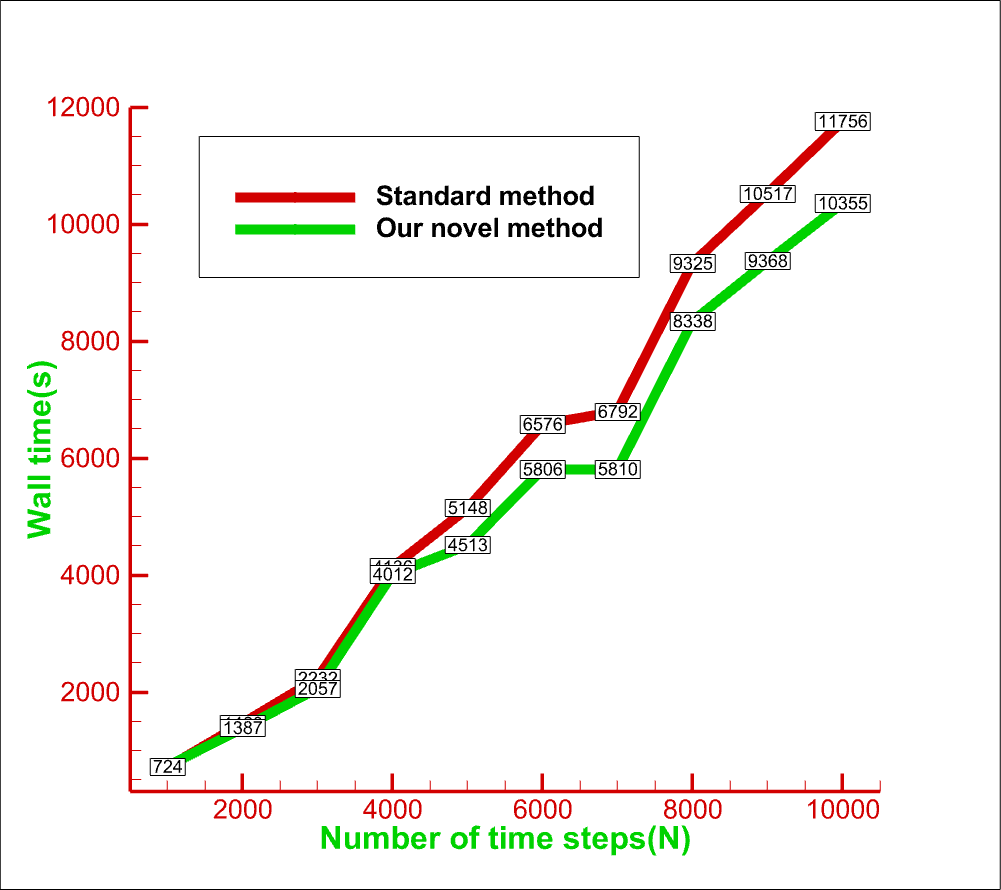}
			\end{minipage}
			\begin{minipage}{0.49\linewidth}
				\centering
				\includegraphics[width=0.9\linewidth]{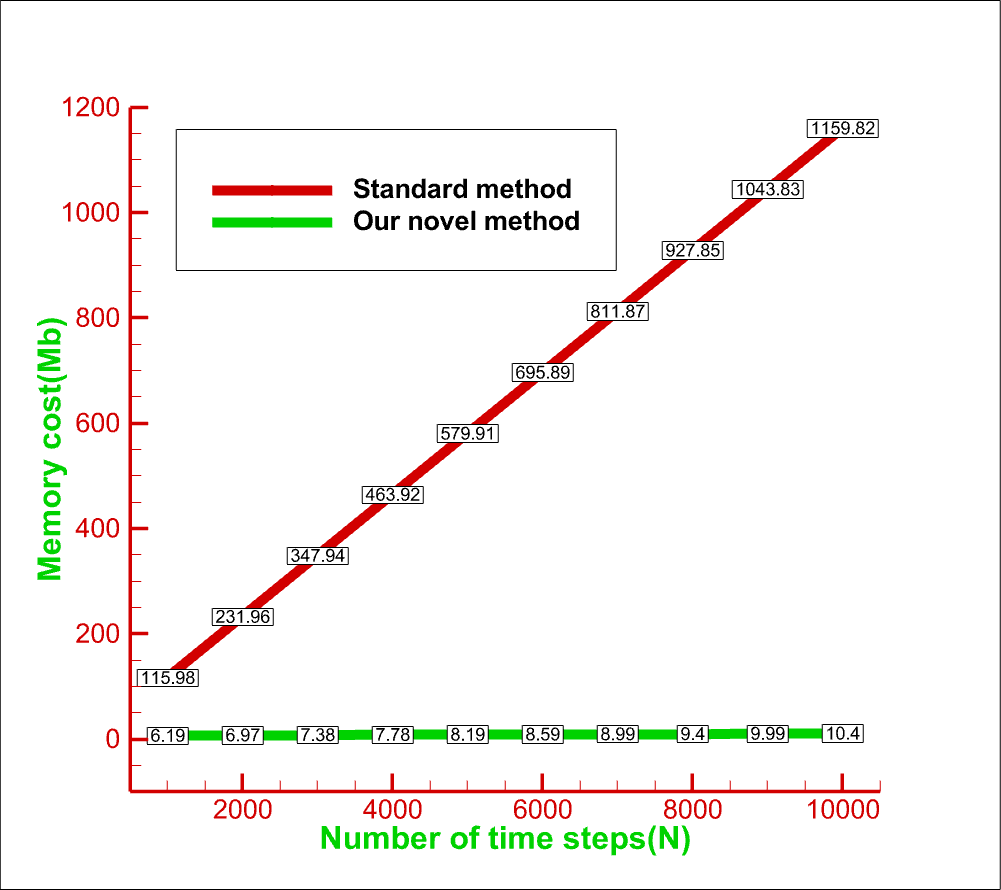}
			\end{minipage}
			\caption{A comparison of wall time and memory costs between the two algorithms is conducted for various time steps , specifically when $h=1/50$ and $\Delta t=10^{-4}$.}
			\label{figure--2}
		\end{figure}
		\end{example}
\begin{example}
	In this example we consider the benchmark problem of planar four-to-one contraction flow to validate the fully discrete scheme proposed in \Cref{section 2,section 4}.  The computational domain is   symmetric, demonstrated in \cref{district}, and of which the upstream and downstream lengths are all 8, and
	the widths of these two stream channels are $ 4:1 $, respectively. We impose fully developed flow boundary conditions at inflow boundary:
	\begin{align*}
		u_1=\frac{3}{8}(1-(\frac{4-y}{4})^2),\qquad u_2=0,
	\end{align*}
and at outflow boundary:
\begin{align*}
u_1=\frac{3}{2}(1-(y-4)^2),\quad u_2=0.
\end{align*} 
 While no-slip conditions are imposed on solid walls
\begin{align*}
	u_1=0,\qquad u_2=0.
\end{align*}
The viscoelastic flow equation is of the following form:
\begin{align*}
	\bm{u}_t-\mu\Delta \bm{u}-\int_{0}^{t}K(t-s)\Delta \bm{u}(s)ds&+(\bm{u}\cdot\nabla)\bm{u}+\nabla p=\bm{0},\\
	\nabla \cdot \bm{u}&=\bm{0},
\end{align*}
where
\begin{align*}
	K(t)= \rho e^{-\delta t},
\end{align*}
and  initial condition  is given:
\begin{align*}
	\bm{u}_0=\bm{0}.
\end{align*}
\begin{figure}[H]
	\centering
	\includegraphics[width=0.9\linewidth]{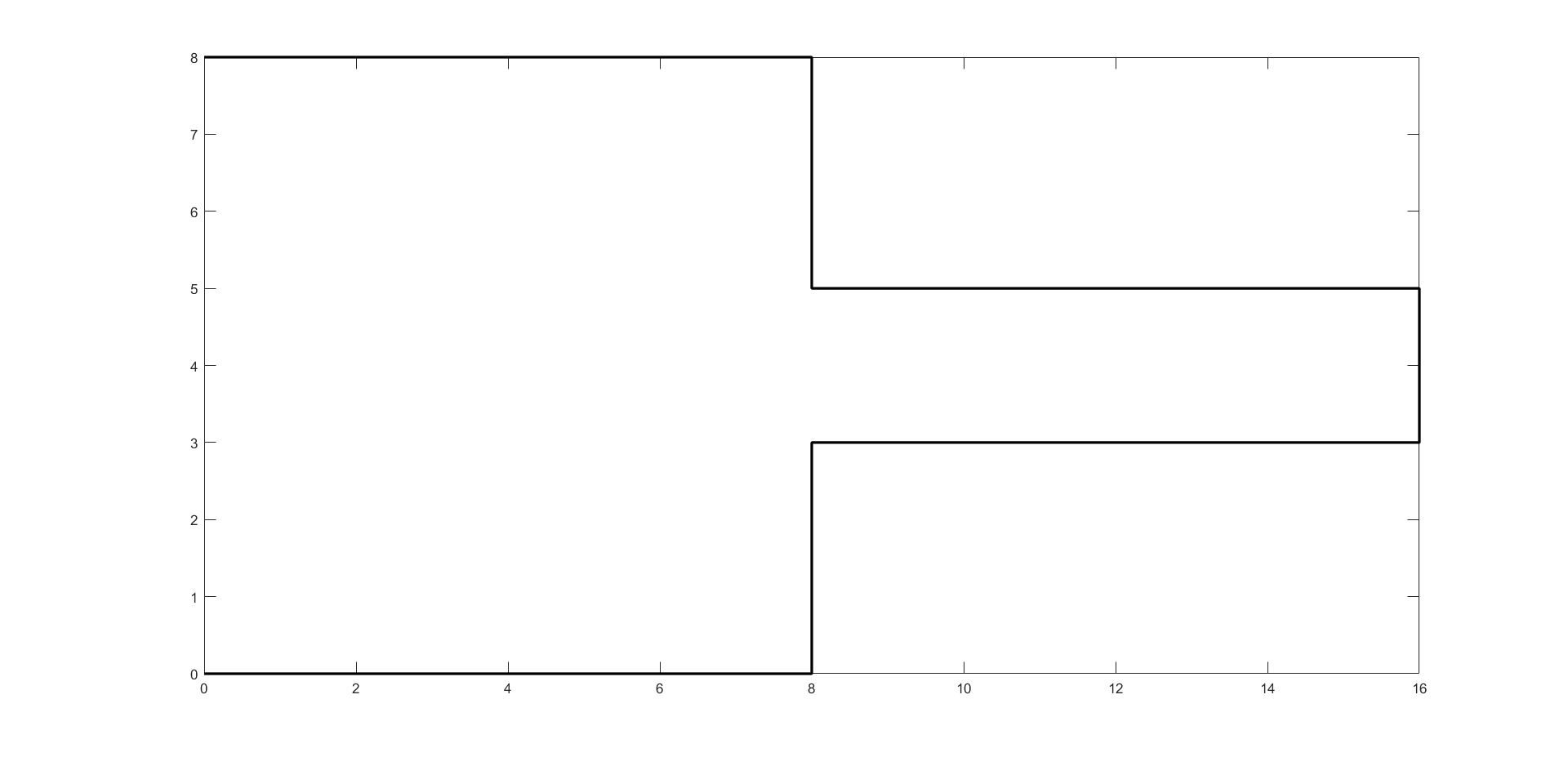}
	\caption{Geometry description of the planar four to one contraction flow domain}
		\label{district}
\end{figure}

The benchmark is tested with viscosity coefficient $\mu=100$, $\rho=1$, and $\delta=100$. We emphasize that the exponential kernel $K(t)=\rho e^{-\delta t}$ is recurrence-friendly and therefore not the most demanding case from the viewpoint of history storage; this example is included primarily as a qualitative flow benchmark rather than as a worst-case memory test. To resolve the corner singularity, we use the refined mesh shown in \Cref{mesh--local}. The contour plots of the velocity components $u_1$ and $u_2$ at the final time $T=1$ are displayed in \Cref{contour--u1,contour--u2} for the conventional finite element method and for the compressed method.

\begin{figure}[H]
	\centering
	\includegraphics[width=0.9\linewidth]{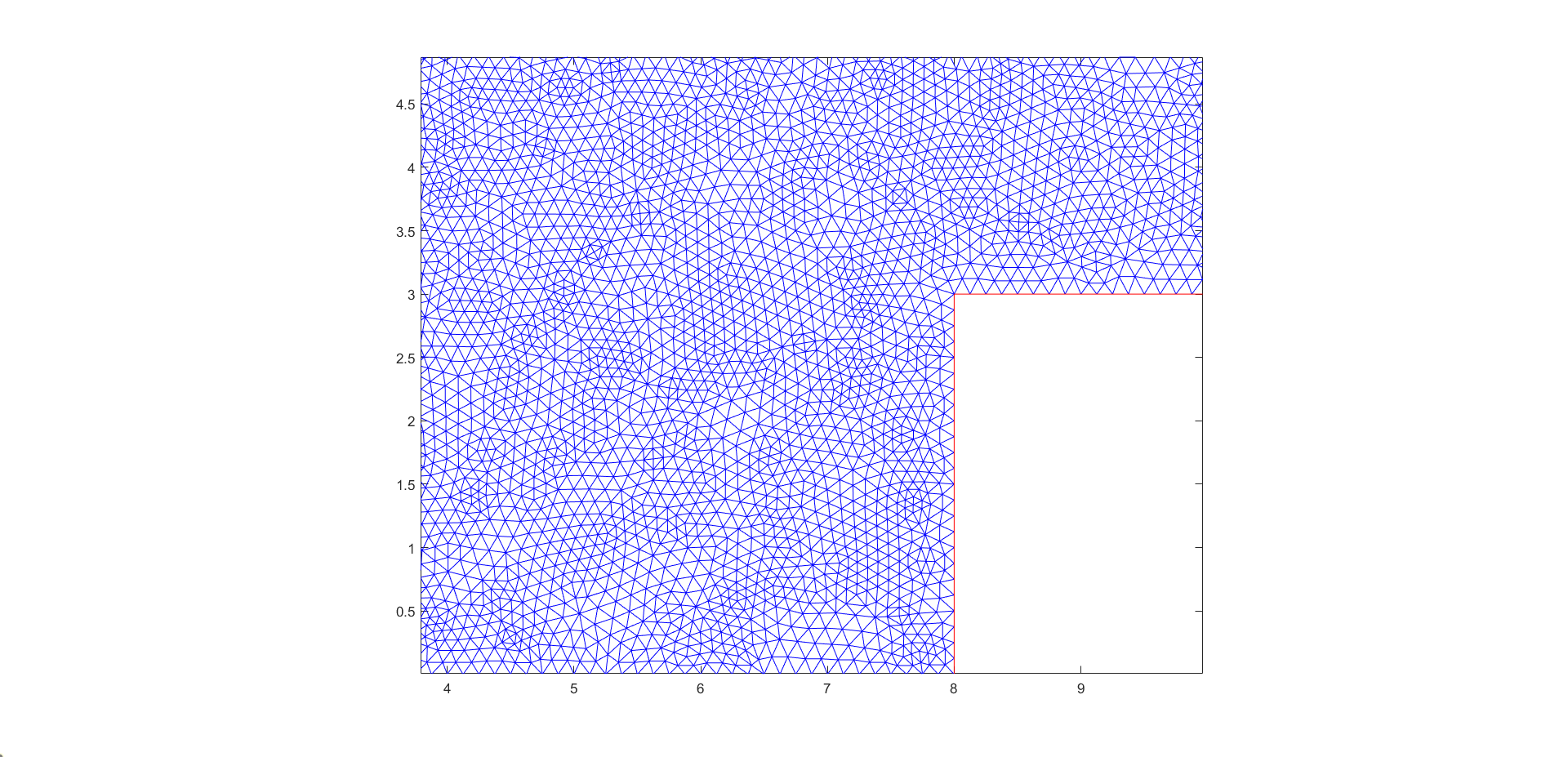}
	\caption{Local views of mesh system of 4 : 1 contraction flow}
	\label{mesh--local}
\end{figure}
	\begin{figure}[h]
	\centering
	\begin{minipage}{0.6\linewidth}
		\centering
		\includegraphics[width=1.0\linewidth]{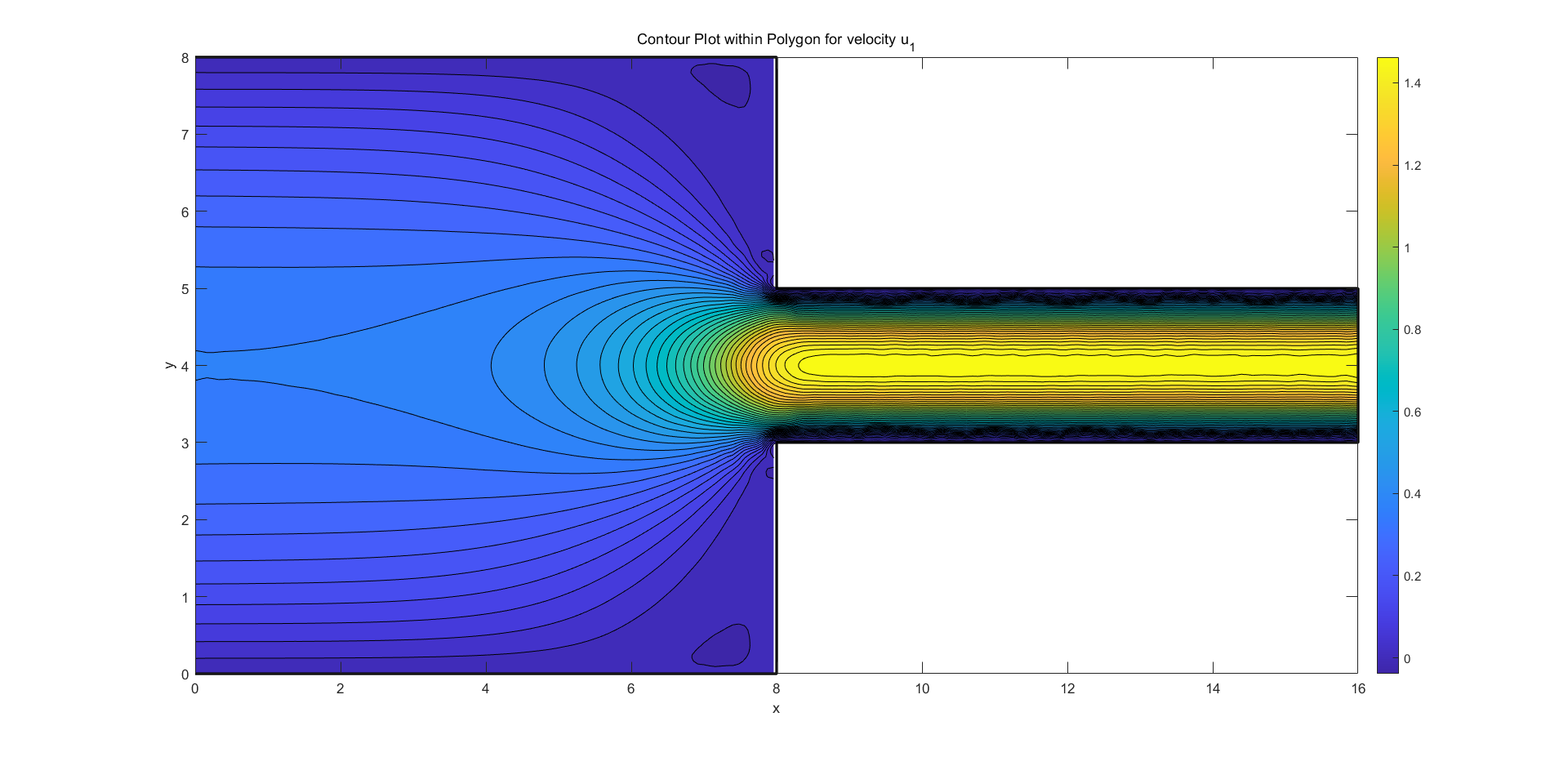}
	\end{minipage}
	\begin{minipage}{0.6\linewidth}
		\centering
		\includegraphics[width=1.0\linewidth]{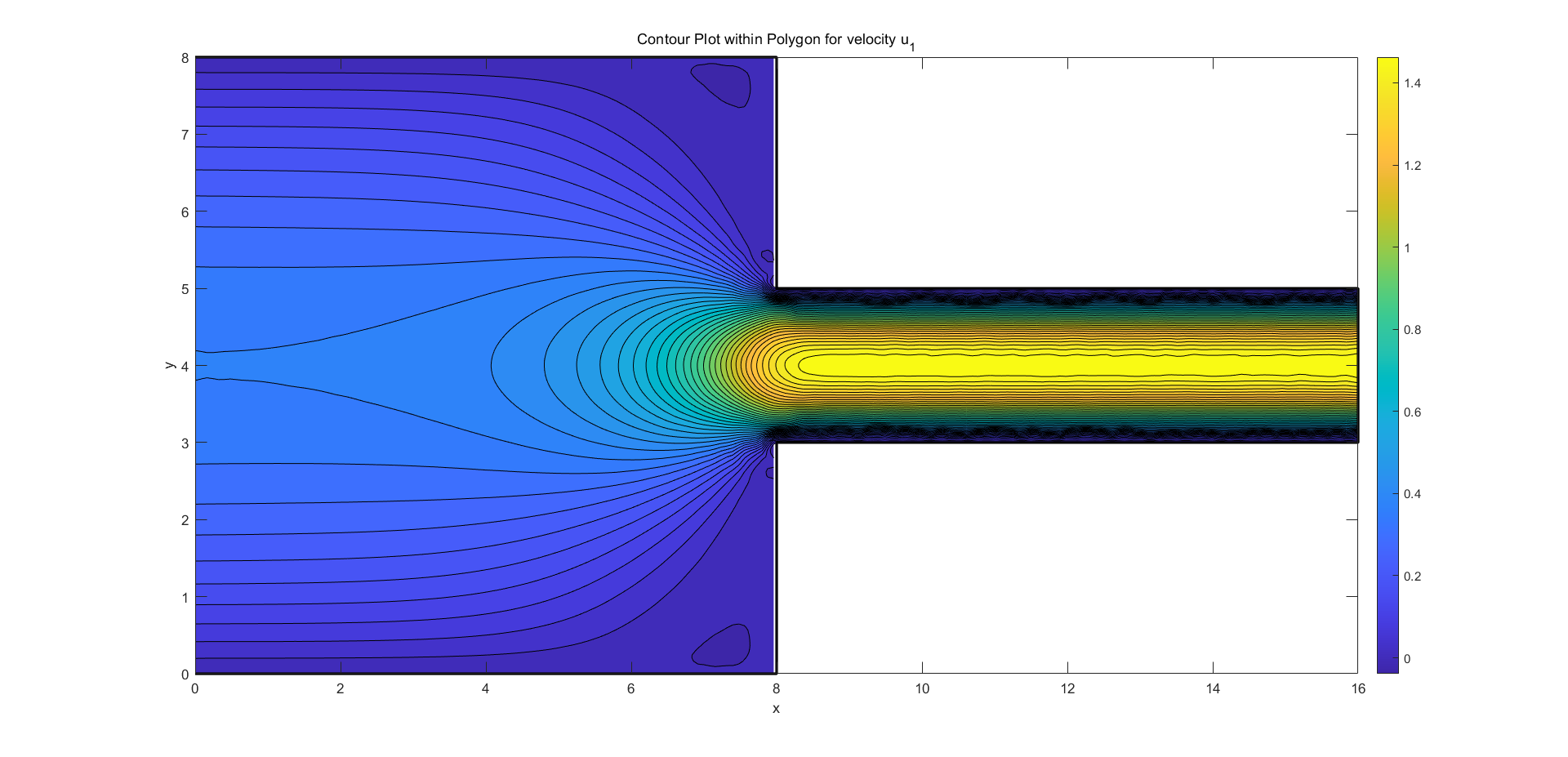}
	\end{minipage}
	\caption{Contour plots of the velocity component $u_1$. The upper panel is obtained by the conventional FEM and the lower panel by the compressed method.}
	\label{contour--u1}
\end{figure}

\begin{figure}[h]
	\centering
	\begin{minipage}{0.6\linewidth}
		\centering
		\includegraphics[width=1.0\linewidth]{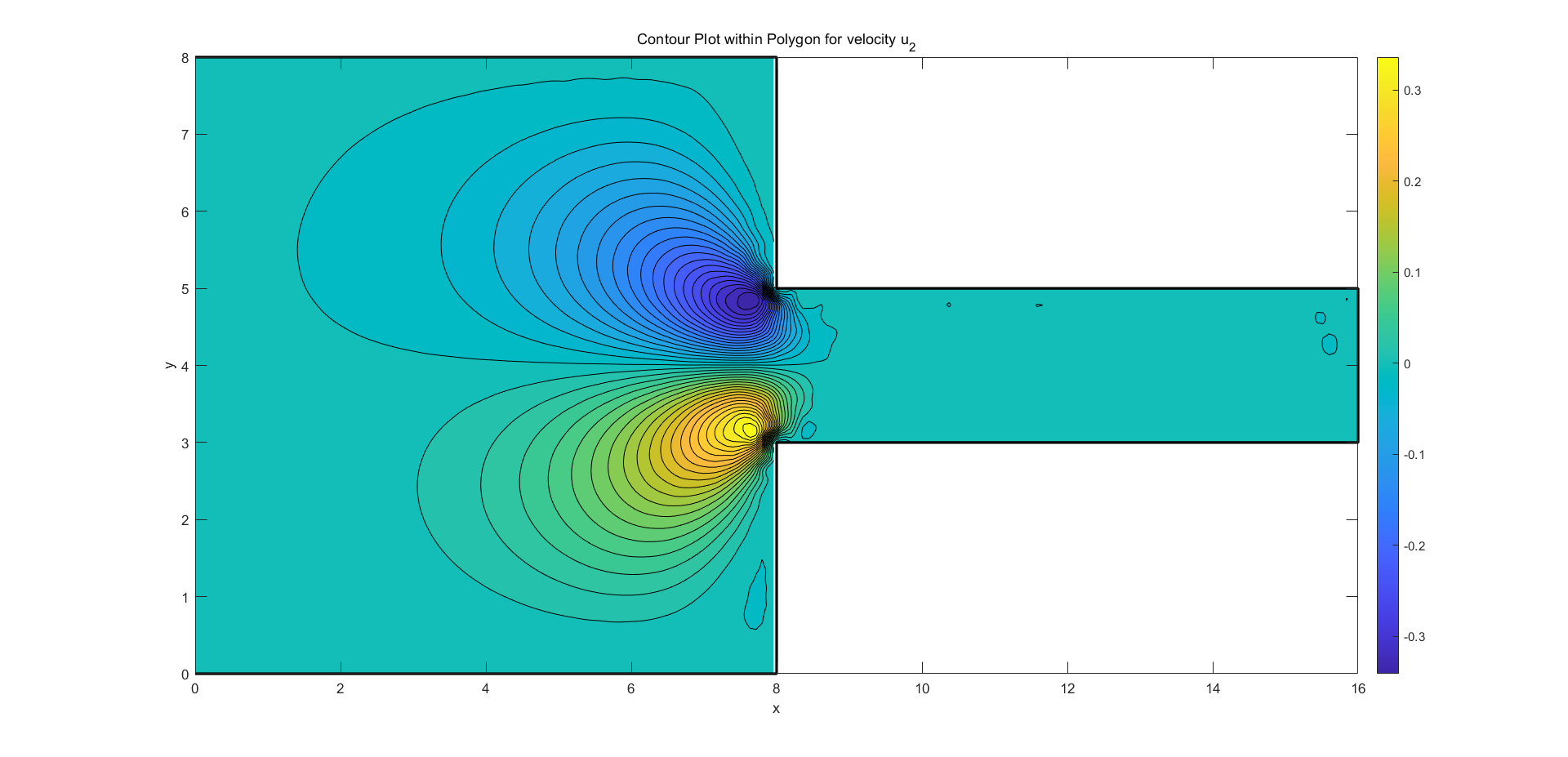}
	\end{minipage}
	\begin{minipage}{0.6\linewidth}
		\centering
		\includegraphics[width=1.0\linewidth]{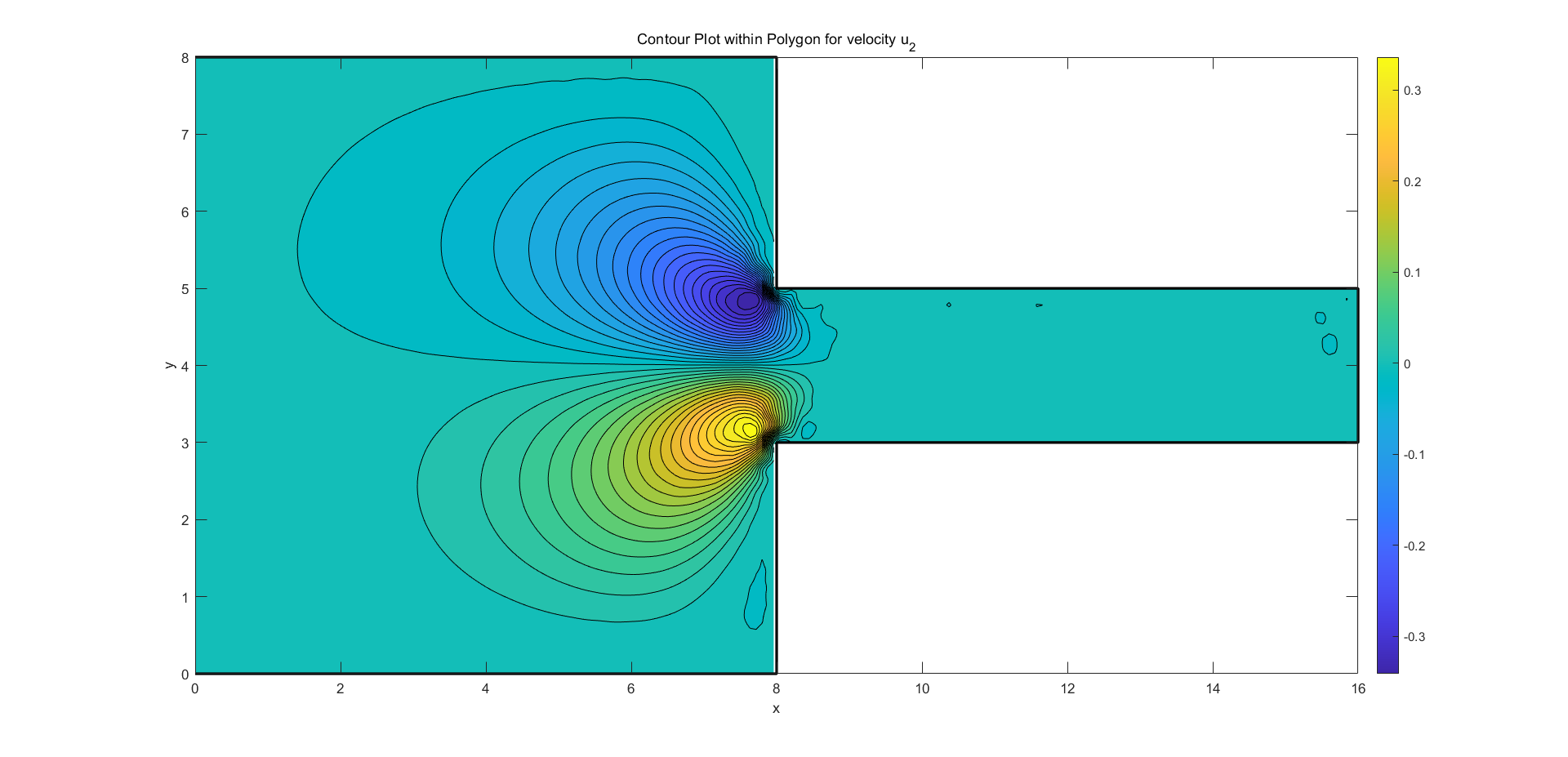}
	\end{minipage}
	\caption{Contour plots of the velocity component $u_2$. The upper panel is obtained by the conventional FEM and the lower panel by the compressed method.}
	\label{contour--u2}
\end{figure}

The two contour pairs are visually consistent, indicating that the compression does not materially change the resolved flow pattern for this benchmark. For a full journal submission, this section should ideally be complemented by quantitative diagnostics such as pressure drop, centerline velocity profiles, vortex size, retained rank, and compression ratio. We record this benchmark here as a first qualitative validation only.

\end{example}

\section{Conclusion and outlook}\label{section 8}

We proposed an incremental-SVD history-compression strategy for mixed finite element discretizations of nonlinear Oldroyd equations with general memory kernels. For nonsingular kernels, the method preserves the baseline accuracy of the fully discrete scheme up to an explicit tolerance-dependent perturbation term. The same compression idea extends to tempered weakly singular kernels discretized by convolution quadrature. The numerical examples show that, for the reported tolerances, the compressed and uncompressed schemes are nearly indistinguishable while the compressed implementation uses substantially less memory.

The present work also has clear limitations. First, the efficiency of the method depends on low numerical rank of the snapshot matrix, and that assumption should be documented empirically through singular-value decay and retained-rank diagnostics in future studies. Second, the current implementation still incurs quadratic dependence on the number of time steps in direct history accumulation; the compression reduces the dependence on the spatial dimension but does not by itself yield an $\mathcal O(N\log N)$-type fast solver. Third, the benchmark section would be strengthened by additional quantitative flow diagnostics and larger-scale tests.

Natural directions for future work therefore include combining the present solution-history compression with fast convolution quadrature or sum-of-exponentials acceleration, developing adaptive tolerance strategies, and extending the analysis and implementation to larger three-dimensional viscoelastic simulations.

\bibliographystyle{plain}
\bibliography{mybib_revised}

\end{document}